\documentclass{amsart}
\usepackage{a4wide,pstricks,bbm}
\usepackage{amsmath,amssymb,latexsym}
\usepackage{epsf,psfrag,epsfig}
\usepackage[english]{babel}
\usepackage[latin1]{inputenc}
\renewcommand{\arraystretch}{1.2}

\newtheorem{theorem}{Theorem}[section]
\newtheorem{prop}[theorem]{Proposition}
\newtheorem{lemma}[theorem]{Lemma}

\newcommand{\EE}{\mathbb{E}}
\newcommand{\PP}{\mathbf{P}}

\newcommand{\nL}{L}

\newcommand{\nX}{{X}}
\newcommand{\nY}{{Y}}
\newcommand{\nW}{{W}}
\newcommand{\nZ}{{Z}}
\newcommand{\nnT}{{T}}
\newcommand{\nnL}{{L}}

\newcommand{\calG}{\ensuremath{\mathcal{G}}}
\newcommand{\Ex}{\ensuremath{\mathrm{Ex}}}

\title[Graph classes with given $3$-connected components]{Graph classes 
with given $3$-connected components: \\asymptotic enumeration and random graphs}
\thanks{A preliminary version of this work appeared in Electr. Notes Discrete Math. (29) 2007, 521--529.}

\author[O. Gim\'enez]{Omer Gim\'enez}
\address{O. Gim\'enez:  Departament de Llenguatges i Sistemes Inform\`atics, Universitat Polit\`ecnica de Catalunya, 08034 Barcelona, Spain}
\email{omer.gimenez@upc.edu}

\author[M. Noy]{Marc Noy}
\address{M. Noy:  Departament de Matem\`atica Aplicada 2, Universitat Polit\`ecnica de Catalunya, 08034 Barcelona, Spain}
\email{marc.noy@upc.edu}

\author[J. Ru\'{e}]{Juanjo Ru\'{e}}
\address{J. Ru\'{e}: Departament de Matem\`atica Aplicada 2, Universitat Polit\`ecnica de Catalunya, 08034 Barcelona, Spain}
\email{juan.jose.rue@upc.edu}

\thanks{Partially supported by  grant MTM2005-08618-C02-01.}

\begin{document}

\maketitle

\begin{abstract}
Consider a family $\mathcal{T}$ of $3$-connected graphs of moderate
growth, and let $\mathcal{G}$ be the class of graphs whose
$3$-connected components are graphs in $\mathcal{T}$. We present a
general framework for analyzing such graphs classes based on
singularity analysis of generating functions, which generalizes
previously studied cases such as planar graphs and series-parallel
graphs. We provide a general result for the asymptotic number of
graphs in $\mathcal{G}$, based on the singularities of the
exponential generating function associated to $\mathcal{T}$. We
derive limit laws, which are either normal or Poisson, for several
basic parameters, including the number of edges, number of blocks
and number of components. For the size of the largest block we find
a fundamental dichotomy: classes similar to planar graphs have
almost surely a unique block of linear size, while classes similar
to series-parallel graphs have only sublinear blocks. This dichotomy
also applies to the size of the largest 3-connected component. For
some classes under study both regimes occur, because of a critical
phenomenon as the edge density in the class varies.
\end{abstract}

\section{Introduction}
Several enumeration problems on planar graphs have been solved
recently. It has been shown~\cite{gn} that the number of labelled
planar graphs with $n$ vertices is asymptotically equal to
\begin{equation}\label{asympt-planar}
        c \cdot n^{-7/2} \cdot \gamma^n n!,
\end{equation}
for suitable constants $c$ and $\gamma$. For series-parallel
graphs~\cite{SP}, the asymptotic estimate is of the form, again for
suitable constants $d$ and $\delta$,
\begin{equation}\label{asympt-sp}
        d \cdot n^{-5/2} \cdot \delta^n n!.
\end{equation}
As can be seen from the proofs in \cite{SP,gn}, the difference in
the subexponential term comes from a different behaviour of the
counting generating functions near their dominant singularities.
Related families of labelled graphs have been studied, like
outerplanar graphs~\cite{SP}, graphs not containing $K_{3,3}$ as a
minor \cite{k33}, and, more generally, classes of graphs closed
under minors \cite{growth}. In all cases where asymptotic estimates
have been obtained, the subexponential term is systematically either
$n^{-7/2}$ or $n^{-5/2}$. The present chapter grew up as an attempt
to understand this dichotomy.

A~\emph{class} of graphs is a family of labelled graphs which is
closed under isomorphism. A~class~$\mathcal{G}$ is \emph{closed} if
the following condition holds: a graph is in $\mathcal{G}$ if and
only if its connected, 2-connected and 3-connected components are
in~$\mathcal{G}$. A~closed class is completely determined by its
3-connected members. The basic example is the class of planar
graphs, but there are others, specially minor-closed classes whose
excluded minors are 3-connected.

In this paper we present a general framework for enumerating
closed classes of graphs. Let $T(x,z)$ be the generating function
associated to the family of 3-connected graphs in a closed class
$\mathcal{G}$, where $x$ marks vertices and $z$ marks edges, and let
$g_n$ be the number of graphs in~$\mathcal{G}$ with~$n$ vertices.
Our first result shows that the asymptotic estimates for $g_n$ depend crucially
on the singular behaviour of $T(x,z)$. For a fixed value of $x$, let
$r(x)$ be the  dominant singularity of $T(x,z)$. If $T(x,z)$ has an
expansion at $r(x)$ in powers of $Z =\sqrt{1-z/r(x)}$ with dominant
term $Z^{5}$, then the estimate for $g_n$ is as in
Equation~(\ref{asympt-planar}); if $T(x,z)$ is either analytic
everywhere or the  dominant term is $Z^{3}$, then the pattern is
that of Equation~(\ref{asympt-sp}).
Our analysis gives a clear explanation of these facts in term of the vanishing of
certain coefficients in singular expansions (Propositions
\ref{proposition:B's}, \ref{prop:singB}, and \ref{proposition:C's}).

There also mixed cases, where
2-connected and connected graphs in $\mathcal{G}$ get different
exponents. And there are \emph{critical} cases too, due to the
confluence of two sources for the dominant singularity, where a
subexponential term $n^{-8/3}$ appears. This is the content of
Theorem~\ref{th:mainresult}, whose proof is based on a careful
analysis of singularities.

Section \ref{se:prelimi} presents technical preliminaries needed
in the paper, and Section \ref{se:asympt} contains the main results.
In Section~\ref{se:laws}, extending the analytic techniques
developed for asymptotic enumeration, we analyze random graphs from
closed classes of graphs. We show that several basic parameters
converge in distribution  either to a normal law or to a Poisson
law. In particular, the number of edges, number of blocks and number
of cut vertices are asymptotically normal with linear mean and
variance. This is also the case for the number of special copies of
a fixed graph or a fixed block in the class. On the other hand, the
number of connected components converges to a discrete Poisson law.

In Section~\ref{se:bloc} we study a key extremal parameter: the size
of the largest block, or the largest 2-connected component. And in
this case we find a striking difference depending on the class of
graphs. For planar graphs there is asymptotically almost surely  a
block of linear size, and the remaining blocks are of order
$O(n^{2/3})$. For series-parallel graphs there is no block of linear
size. This also applies more generally to the classes considered in
Theorem~\ref{th:mainresult}. A similar dichotomy occurs when
considering the size of the largest 3-connected component. This is
proved using the techniques developed by Banderier et al.
\cite{airy} for analyzing largest components in random maps. For
planar graphs we prove the following precise result in
Theorem~\ref{th:largest-block}. If $\nX_n$ is the size of the
largest block in random planar graphs with $n$ vertices, then
$$
    \PP\left(\nX_n = \alpha  n + x n^{2/3}\right) \sim n^{-2/3} c g(c x),
$$
where $\alpha \approx 0.95982$ and  $ c \approx  128.35169$ are
well-defined analytic constants, and $g(x)$ is the so called Airy
distribution of the map type, which is a particular  instance of a stable
law of index $3/2$. Moreover, the size of the second largest block
is $O(n^{2/3})$. The giant block is uniformly distributed among the
planar 2-connected graphs with the same number of vertices, hence
according to the results in \cite{bender} it has about $2.2629 \cdot
0.95982\,n = 2.172\, n$ edges, again with deviations of order
$O(n^{2/3})$ (the deviations for the normal law are of order $n^{1/2}$, but
the $n^{2/3}$ term coming from the Airy distribution dominates).
We remark that the size of the largest block has been analyzed too in \cite{kostas}
using different techniques. The main improvement with respect to
\cite{kostas} is that we are able to obtain a precise limit distribution.

With respect to the largest 3-connected component in a
random planar graph, we show that it follows an Airy distribution and has $\eta n$ vertices and
$\zeta n$ edges, where $\eta \approx 0.7346 $ and $\zeta \approx
1.7921 $ are again well-defined. This is technically more involved since we have to analyze the composition of two Airy laws and different probability distributions in 2-connected graphs.

The picture that emerges for random planar graphs   is the
following. Start with a large 3-connected planar graph $M$ (or the
skeleton of a polytope in the space if one prefers a more geometric
view), and perform the following operations. First edges of $M$ are
substituted by small blocks with a distinguished oriented edge,
giving rise to the giant block
$L$; then small connected graphs are attached to some of the
vertices of $L$, which become cut vertices, giving rise to the
largest connected component~$C$. As we show later, $C$
has size $n-O(1)$. This model can be made
more precise and will be the subject of future research.

An interesting open question is whether there are other parameters
besides the size of the largest block (or largest 3-connected
component) for which planar graphs and series-parallel graphs differ
in a qualitative way. We remark that with respect to the largest
component there is no qualitative difference. This is also true for the degree
distribution \cite{degrees}. If $d_k$ is the probability that a
given vertex has degree $k>0$, then in both cases it can be shown
that the $d_k$ decay as $c\cdot n^\alpha q^k$, where $c,\alpha$ and
$q$ depend on the class under consideration \cite{degrees}.

In Section~\ref{se:examples} we apply the previous machinery to the
analysis of several classes of graphs closed under minors, including planar graphs and
series-parallel graphs. Whenever the generating function $T(x,z)$
can be computed explicitly, we obtain precise asymptotic estimates
for the number of graphs $g_n$, and limit laws for the main
parameters. In particular we determine the asymptotic probability of
a random graph being connected,  the constant $\kappa$ such that
the expected number of edges is asymptotically $\kappa n$, and
other fundamental constants.

Our techniques allow also to study graphs with a given density, or
average degree. To fix ideas, let $g_{n,\lfloor \mu n \rfloor}$ be
the number of planar graphs with $n$ vertices and $\lfloor \mu n
\rfloor$ edges: $\mu$~is the edge density and $2\mu$ is the average
degree. For $\mu \in (1,3)$, a precise estimate for $g_{n,\lfloor
\mu n \rfloor}$ can be obtained using a local limit
theorem~\cite{gn}. And parameters like the number of components or
the number of blocks can be analyzed too when the edge density
varies. It turns out that the family of planar graphs with density
$\mu \in (1,3)$ shares the main characteristics of planar graphs.
This is also the case for series-parallel graphs, where $\mu \in
(1,2)$ since maximal graphs in this class have only $2n-3$ edges.
In Section~\ref{se:critical} we show  examples of \emph{critical phenomena} by
a suitable choice of the family $\mathcal{T}$ of 3-connected graphs.
In the associated closed class $\mathcal{G}$, graphs below a
critical density $\mu_0$ behave like series-parallel  graphs, and
above $\mu_0$ they behave like planar graphs, or conversely. We even
have examples with more than one critical value.

We remark that graph classes with given 3-connected components are analyzed also in \cite{grammar} and \cite{gagarin}, where the emphasis is on combinatorial decompositions rather than asymptotic analysis.

\section{Preliminaries}\label{se:prelimi}

Generating functions are of the exponential type, unless we say
explicitly the contrary. The partial derivatives of $A(x,y)$ are
written $A_x(x,y)$ and  $A_y(x,y)$. In some cases the derivative
with respect to $x$ is written $A'(x,y)$. The second derivatives are
written $A_{xx}(x,y)$, and so on. By a.a.s.\ we mean asymptotically
almost surely, which in our case means a property of random graphs
whose probability tends to $1$ as $n$ goes to infinity.

The decomposition of a graph into connected components, and of a
connected graph into blocks (2-connected components) are well known.
We also need the decomposition of a 2-connected graph decomposes
into 3-connected components \cite{Tutte}. A 2-connected graph is
built by series and parallel compositions and 3-connected graphs in
which each edge has been substituted by a block; see below the definition of networks.

A class of labelled graphs~$\mathcal{G}$ is \emph{closed} if a graph
$G$ is in $\mathcal{G}$ if and only if the connected, 2-connected
and 3-connected components of $G$ are in~$\mathcal{G}$. A~closed
class is completely determined by the family $\mathcal{T}$ of its
3-connected members. Let $g_{n}$ be the number of graphs in
$\mathcal{G}$ with $n$ vertices, and let $g_{n,k}$ be the number of
graphs with $n$ vertices and $k$ edges. We define similarly
$c_n,b_n,t_n$ for the number of connected, 2-connected and
3-connected graphs, respectively, as well as the corresponding
$c_{n,k},b_{n,k},t_{n,k}$. We introduce the EGFs
 $$
    G(x,y) = \sum_{n,k} g_{n,k} \, y^k  {x^n \over n!},
 $$
and similarly for $C(x,y)$ and $B(x,y)$. When $y=1$ we recover the
univariate EGFs
$$
    B(x) = \sum b_n {x^n \over n!}, \qquad C(x) = \sum c_n {x^n \over n!},
    \qquad G(x) = \sum g_n {x^n \over n!}.
$$
The following equations reflect the decomposition into connected
components and 2-connected components:
\begin{equation}\label{eq:GCB2}
    G(x,y) = \exp(C(x,y)), \qquad x C'(x,y)    =
    x\exp\left(B'(x C'(x,y),y)
    \right),
\end{equation}
In the first decomposition, one must notice that a general graph is
simply a \emph{set} of labelled connected graphs, hence the equation
$G(x,y) = \exp(C(x,y))$. The second decomposition is a bit more involved.
The EGF $x C'(x,y)$ is associated to the family of connected graphs
with rooted at a vertex. Then, the second equation in~(\ref{eq:GCB2})
says that a connected graph with a rooted vertex is obtained from a
set of rooted $2$-connected graphs (where the root bears no label), in
which we substitute each vertex by a connected graph with a rooted vertex
(the roots allow us to recover the graph from its constituents). We also
define
$$
    T(x,z) = \sum_{n,k} t_{n,k} \, z^n {x^n\over n!},
$$
where the only difference is that the variable for marking edges is
now~$z$. This convention is useful and will be maintained throughout
the paper.

A~\emph{network} is a graph with two distinguished vertices, called
\emph{poles}, such that the graph obtained by adding an edge between
the two poles is 2-connected. Moreover,  the two poles are not
labelled. Networks are the key technical device for encoding the
decomposition of 2-connected graphs into 3-connected components.
We distinguish between three kinds of networks. A network is \textit{series} if it is obtained from a cycle C with a distinguished edge $e$, whose endpoints become the
poles, and every edge different from $e$ is replaced by a network. Equivalently, when removing the root edge if present, the resulting graph is not 2-connected. A network
is \textit{parallel} if it is obtained by gluing two or more networks, none of them containing the root edge, along the common poles. Equivalently, when the two poles are a 2-cut of the network. Finally, an \textit{h-network} is obtained from a 3-connected graph $H$ rooted at an oriented edge, by replacing every edge of $H$ (other than the root) by an arbitrary network. Trakhtenbrot \cite{trak} showed that a network is either series, parallel or an h-network, and Walsh \cite{walsh} translated this fact into generating functions as we show next.

Let $D(x,y)$ be the GF associated to networks, where again $x$ and $y$
mark vertices and edges. Then $D=D(x,y)$ satisfies~(see
\cite{bender}, who draws on \cite{trak,walsh})
\begin{equation}\label{eq:D}
  {2 \over x^2} T_z(x,D) - \log\left({1+D \over 1+y}\right) +
  {xD^2 \over 1+xD} = 0,
\end{equation}
and $B(x,y)$ is related to $D(x,y)$ through
\begin{equation}\label{eq:BD}
B_y(x,y) = {x^2 \over 2}  \left( {1+D(x,y)  \over 1+y} \right).
\end{equation}
For future reference, we set
\begin{equation}\label{eq:phi}
\Phi(x,z)= \frac{2}{x^{2}}\,T_z(x,z)
-\log\left(\frac{1+z}{1+y}\right) +\frac{xz^{2}}{1+xz},
\end{equation}
so that Equation~(\ref{eq:D}) is written in the form $\Phi(x,D)=0$,
for a given value of $y$. By integrating (\ref{eq:BD}) using the
techniques developed in \cite{gn}, we obtain an explicit expression
for $B(x,y)$ in terms of $D(x,y)$ and $T(x,z)$ (see the first part
of the proof of Lemma~5 in \cite{gn}):
\begin{eqnarray}\label{eq:Bexplicit}
B(x,y)&=& T(x,D(x,y))
-\frac{1}{2}xD(x,y)+\frac{1}{2}\log(1+xD(x,y))+ \\
&& \frac{x^{2}}{2}\left(D(x,y)+\frac{1}{2}D(x,y)^2+
(1+D(x,y))\log\left(\frac{1+y}{1+D(x,y)}\right)\right) . \nonumber
\end{eqnarray}
This relation is valid for every closed defined in terms of
3-connected graphs, and can be proved in a more combinatorial
way~\cite{grammar} (see also \cite{gagarin}).

We use singularity analysis for obtaining asymptotic estimates; the main reference here is \cite{FlajoletSedgewig:analytic-combinatorics}.
The singular expansions we encounter in this paper are always of the form
$$
f(x) = f_0 + f_2X^2 + f_4X^4 + \cdots  + f_{2k}X^{2k} + f_{2k+1}X_{2k+1} + O(X_{2k+2}),
$$
where $X = \sqrt{1 - x/\rho}$. That is, $2k + 1$ is the smallest odd integer $i$ such that $f_i \ne 0$. The even powers of X are analytic functions and do not contribute to the asymptotics of $[x^n]f(x)$.
The number $\alpha = (2k+1)/2$ is called the \textit{singular exponent},
and by the transfer theorem \cite{FlajoletSedgewig:analytic-combinatorics} we obtain the estimate
$$
[x^n]f(x) \sim c \cdot n^{\alpha-1} \rho{-n},
$$
where $c = f_{2k+1}/\Gamma(-\alpha)$.

We assume that, for a fixed value of $x$, $T(x,z)$ has a unique
dominant  singularity  $r(x)$, and that there is a singular
expansion near $r(x)$ of the form
\begin{equation}\label{eq:singT}
    T(x,z) = \sum_{n\ge n_0} T_n(x) \left(1- {z \over r(x)}\right)^{n/\kappa},
\end{equation}
where $n_0$ is an integer, possibly negative, and the functions
$t_n(x)$ and $r(x)$ are analytic. This is a rather general
assumption, as it includes singularities coming from algebraic and
meromorphic functions.

The case when $\mathcal{T}$ is empty (there are no 3-connected
graphs) gives rise to the class of series-parallel graphs.
It is shown in~\cite{SP} that, for a fixed value $y=y_0$, $D(x,y_0)$
has a unique dominant singularity $R(y_0)$. This is also true for
arbitrary $\mathcal{T}$, since adding 3-connected graphs can only
increase the number of networks.

\section{Asymptotic enumeration}\label{se:asympt}

Throughout the rest of the chapter we assume that $\mathcal{T}$ is a
family of 3-connected graphs whose GF $T(x,z)$ satisfies the
requirements described in Section~\ref{se:prelimi}. We assume that a
singular expansion like (\ref{eq:singT}) holds, and we let $r(x)$ be
the dominant singularity of $T(x,z)$, and $\alpha$ the singular
exponent.

Our main result gives precise asymptotic estimates for $g_n, c_n,
b_n$ depending on the singularities of $T(x,z)$. Cases (1) and (2)
in the next statement can be considered as generic, whereas (1) and
(2.1) are those encountered in `natural' classes of graphs. The two
situations in case~(3) come from critical conditions, when two
possible sources of singularities coincide. This is the reason for
the unusual exponent $-8/3$, which comes from a singularity of
cubic-root type instead of the familiar square-root type.

\begin{theorem}\label{th:mainresult}
Let $\mathcal{G}$ be a closed family of graphs, and let $T(x,z)$ be
the GF of the family of 3-connected graphs in $\mathcal{G}$. In all
cases $b,c,g,R,\rho$ are explicit positive constants and $\rho < R$.
\begin{itemize}
\item[(1)] If $T_z(x,z)$ is either analytic or has singular
exponent $\alpha<1$, then
$$
b_n\sim b\, n^{-5/2} R^{-n} n!, \qquad c_n\sim c\, n^{-5/2}
\rho^{-n} n!, \qquad g_n\sim g\, n^{-5/2} \rho^{-n} n!
$$
\item[(2)] If $T_z(x,z)$ has singular exponent $\alpha=3/2$, then one of
the following holds:
\smallskip
\begin{itemize}
\item[(2.1)]
$ \displaystyle b_n\sim b\, n^{-7/2} R^{-n} n!, \qquad c_n\sim c\,
n^{-7/2} \rho^{-n} n!, \qquad g_n\sim g\, n^{-7/2} \rho^{-n} n! $
\item[(2.2)]
$ \displaystyle b_n\sim b\, n^{-7/2} R^{-n} n!, \qquad c_n\sim c\,
n^{-5/2} \rho^{-n} n!, \qquad g_n\sim g\, n^{-5/2} \rho^{-n} n! $
\item[(2.3)]
$ \displaystyle b_n\sim b\, n^{-5/2} R^{-n} n!, \qquad c_n\sim c\,
n^{-5/2} \rho^{-n} n!, \qquad g_n\sim g\, n^{-5/2} \rho^{-n} n! $
\end{itemize}

\bigskip
\item[(3)]
If $T_z(x,z)$ has singular exponent $\alpha=3/2$, and in addition a
critical condition is satisfied, one of the following holds:
\smallskip
\begin{itemize}
\item[(3.1)]
$ \displaystyle b_n\sim b\, n^{-8/3} R^{-n} n!, \qquad c_n\sim c\,
n^{-5/2} \rho^{-n} n!, \qquad g_n\sim g\, n^{-5/2} \rho^{-n} n! $
\item[(3.2)]
$ \displaystyle b_n\sim b\, n^{-7/2} R^{-n} n!, \qquad c_n\sim c\,
n^{-8/3} \rho^{-n} n!, \qquad g_n\sim g\, n^{-8/3} \rho^{-n} n! $
\end{itemize}
\end{itemize}

\end{theorem}

Using the Transfer Theorems~\cite{FlajoletSedgewig:analytic-combinatorics}, the
previous theorem is a direct application of the following analytic
result for $y=1$. We prove it for arbitrary values of $y=y_0$, since
this has important consequences later on.

\begin{theorem}\label{th:mainresult2}
Let $\mathcal{G}$ be a closed family of graphs, and let $T(x,z)$ be
the GF of the family of 3-connected graphs in $\mathcal{G}$.

For a fixed  value $y=y_0$, let $R=R(y_0)$ be the dominant singularity of
$D(x,y_0)$, and  let $D_0 = D(R,y_0)$.

\begin{itemize}
\item[(1)] If~$T_z(x,z)$ is either analytic or has singular
exponent $\alpha<1$ at $(R,D_0)$, then $B(x,y_0)$, $C(x,y_0)$ and
$G(x,y_0)$ have singular exponent $3/2$.
\bigskip
\item[(2)] If $T_z(x,z)$ has singular exponent $\alpha=3/2$ at $(R,D_0)$,
then one of the following holds:
\smallskip
\begin{itemize}
\item[(2.1)]
$B(x,y_0), C(x,y_0)$ and $G(x,y_0)$ have singular exponent $5/2$.
\item[(2.2)]
$B(x,y_0)$ has singular exponent $5/2$, and $C(x,y_0), G(x,y_0)$
have singular exponent $3/2$.
\item[(2.3)]
$B(x,y_0), C(x,y_0)$ and $G(x,y_0)$ have singular exponent $3/2$.
\end{itemize}
\bigskip
\item[(3)]
If $T_z(x,z)$ has singular exponent $\alpha=3/2$ at $(R,D_0)$, and
in addition a critical condition is satisfied for the singularities
of  either $B(x,y)$ or $C(x,y)$, then one of the following holds:
\smallskip
\begin{itemize}
\item[(3.1)]
$B(x,y_0)$ has singular exponent $5/3$, and $C(x,y_0), G(x,y_0)$
have singular exponent $3/2$.
\item[(3.2)]
$B(x,y_0)$ has singular exponent $5/2$, and $C(x,y_0), G(x,y_0)$
have singular exponent $5/3$.

\end{itemize}

\end{itemize}
\end{theorem}


The rest of the section is devoted to the proof of Theorem
\ref{th:mainresult2}, which implies Theorem~\ref{th:mainresult}.
First we study the singularities of $B(x,y)$, which is the most
technical part. Then we study the singularities of $C(x,y)$ and
$G(x,y)$, which are always of the same type since $G(x,y)= \exp
{(C(x,y))}$.

\subsection{Singularity analysis of $B(x,y)$}\label{subsec:singB}

From now on, we assume that $y=y_0$ is a fixed value, and let
$D(x)=D(x, y_0)$. Recall from Equation (\ref{eq:phi}) that $D(x)$
satisfies $\Phi(x,D(x))=0$, where
$$
\Phi(x,z)= \frac{2}{x^{2}} T_z(x,z)
-\log\left(\frac{1+z}{1+y_0}\right) +\frac{xz^{2}}{1+xz}.
$$
Since a 3-connected graph has at least four vertices,  $T(x,z)$ is
$O(x^4)$. If follows that $D(0) = y_0$ and  $\Phi_z(0,0) =
-1/(1+y_0)<0$. It follows from the implicit function theorem  that
$D(x)$ is analytic at $x=0$.

The next result shows that $D(x)$ has a positive  singularity $R$ and that
$D(x)$ is finite at $R$.

\begin{lemma} With the previous assumptions,  $D(x)$ has a positive singularity
$R=R(y)$, and $D(R)$ is also finite.
\end{lemma}

\begin{proof} We first show that $D(x)$ has a finite singularity. Consider the
family of networks without 3-connected components, which corresponds
to  series-parallel networks, and let $D_{\emptyset}(x,y)$ be the
associated GF. It is shown in~\cite{SP} that the radius of
convergence $R_{\emptyset}(y_0)$ of $D_{\emptyset}(x, y_0)$ is
finite for all $y>0$. Since the set of networks enumerated by
$D(x,y_0)$ contains the networks without three-connected components,
it follows that $D_{\emptyset}(x,y) \le D(x,y)$ and $D(x)$ has a
finite singularity $R(y_0)\leq R_{\emptyset}(y_0)$.

Next we show that $D(x)$ is finite at its dominant singularity $R =
R(y_0)$. Since $R$ is the smallest singularity and $\Phi_z(0,0)<0$,
we have $\Phi_z(x, D(x)) < 0$ for $0\le x<R$. We also have
$\Phi_{zz}(x,z) > 0$ for $x,z>0$. Indeed, the first summand in
$\Phi$ is a series with positive coefficients, and all its
derivatives are positive; the other two terms have with second
derivatives $1/(1+z)^2$ and $2x/(1+xz)^3$, which are also positive.
As a consequence, $\Phi_z(x,D(x))$ is an increasing function and
$\lim_{x\to R^-} \Phi(x,D(x))$ exists and is finite. It follows that
$D(R)$ cannot go to infinity, as claimed.
\end{proof}

Since $R$ is the smallest singularity of $D(x)$, $\Phi(x,z)$ is
analytic for all $x<R$ along the curve defined by $\Phi(x,D(x))=0$.
For $x,z>0$ it is clear that $\Phi$ is analytic at $(x,z)$ if and
only if $T(x,z)$ is also analytic. Thus $T(x,z)$ is also analytic
along the curve $\Phi(x,D(x))=0$ for $x<R$. As a consequence, the
singularity $R$ can only have two possible sources:
\begin{itemize}
\item[(a)] A branch-point $(R, D_0)$ when solving  $\Phi(x,z)=0$,
that is, $\Phi$ and $\Phi_z$ vanish at  $(R, D_0)$.

\item[(b)] $T(x,z)$ becomes singular at $(R, D_0)$, so that  $\Phi(x,z)$
is also singular.
\end{itemize}

Case (a) corresponds to case (1) in Theorem~\ref{th:mainresult2}.
For case (b) we assume that the singular exponent of $T(x,z)$ at the
dominant singularity is $5/2$, which corresponds to families of
3-connected graphs coming for 3-connected planar maps, and related
families of graphs. We could allow more general types of singular exponents but they do not appear in the main examples we have analyzed.

The typical situation is case (2.1) in
Theorem~\ref{th:mainresult2}, but (2.2) and (2.3) are also possible.
It is also possible to have a critical situation, where (a) and (b)
both hold, and this leads to case (3.1): this is treated at the end
of this subsection. Finally, a confluence of singularities may also
arise when solving equation
$$
x C'(x,y)= x\exp\left(B'(x C'(x,y),y)\right).
$$

\medskip
\noindent
Indeed the singularity may come from: (a) a branch point when solving the previous equation; or (b) $B(x,y)$ becomes singular at $\rho(y) C'(\rho(y),y)$, where $\rho(y)$ is the singularity of $C(x,y)$.
When the two sources (a) and (b) for the singularity coincide, we are in case (3.2).
This is treated at the end of
Section~\ref{subsec:singC}.

\subsubsection{$\Phi$ has a branch-point at $(R,D_{0})$}
We assume that $\Phi_z(R, D_0)=0$ and that $\Phi$ is analytic at
$(R, D_0)$. We have seen that $\Phi_{zz}(x,z) > 0$ for $x,z>0$.
Under these conditions, $D(x)$ admits a singular expansion near $R$
of the form
\begin{equation}\label{eq:singD}
D(x) = D_{0}+D_{1}X+D_{2}X^{2}+D_{3}X^{3}+O(X^{4}),
\end{equation}
where $X=\sqrt{1-x/R}$, and $D_1=-\sqrt{2R \Phi_x(R,
D_0)/\Phi_{zz}(R, D_0)}$
(see~\cite{FlajoletSedgewig:analytic-combinatorics}). We remark that
$R$ and the $D_{i}$'s depend implicitly on $y_{0}$.

In the next result we find an explicit expression for $D_1$, which is the dominant term in (\ref{eq:singD}).
This puts into perspective the result found in \cite{SP} for series-parallel graphs, where it was shown that $D_1 <0$ for that class.
\begin{prop}\label{proposition:D's}
Consider the singular expansion~(\ref{eq:singD}). Then $D_1 < 0$ is
given by
\begin{equation*}
  D_1 = - \left(\frac{\displaystyle 2R T_{xz}
  -4T_z+\frac{R^3D_{0}^2}{(1+RD_0)^2}}
  {\displaystyle \frac{R^2}{2(1+D_0)^2}+\frac{R^{3}}{\left(1+R D_0\right)^3}+
  T_{zzz}}
  \right)^{1/2},
\end{equation*}
where the partial derivatives of $T$ are evaluated at $(R, D_0)$.
\end{prop}
\begin{proof} We plug the expansion (\ref{eq:singD}) inside (\ref{eq:phi}) and
extract work out the undetermined coefficients $D_i$. The expression
for $D_1$ follows from a direct computation of $\Phi_x$ and
$\Phi_{zz}$, and evaluating at $(R, D_0)$.
To show that $D_{1}$ does not vanish, notice that
$$
2R  T_{x z} -4 T_z = R^3{\partial \over \partial x}
\left(\frac{2}{x^2}T_z(x,z)\right).  
$$

This is positive since   $2/x^2 T_z$ is a series with positive
coefficients. Since $R, D_0 >0 $, the remaining term in  the
numerator inside the square root is clearly positive, and so is the
denominator. Hence $D_1 <0$.
\end{proof}

From the singular expansion of $D(x)$ and the explicit
expression~(\ref{eq:Bexplicit}) of $B(x,y_{0})$ in terms of
$D(x,y_0)$, it is clear that $B(x)=B(x,y_{0})$ also admits a
singular expansion at the same singularity $R$ of the form
\begin{equation}\label{eq:singB}
B(x) = B_{0}+B_{1}X+B_{2}X^{2}+B_{3}X^{3}+O(X^{4}).
\end{equation}
The next result shows that the singular exponent of $B(x)$ is $3/2$,
as claimed. Again, the fact that $B_1=0$ and $B_3 >0$ explains the results found in \cite{SP} for series-parallel graphs.

\begin{prop}\label{proposition:B's}
Consider the singular expansion~(\ref{eq:singB}). Then  $B_{1}=0$
and $B_{3}>0$ is given by
\begin{equation}\label{eq:B3}
  B_3 =
  \frac{1}{3}\left(4T_z-2R T_{xz} -\frac{R^{3}D_{0}^2}{\left(1+RD_{0}\right)^2}\right)D_{1},
\end{equation}
where the partial derivatives of $T$ are evaluated in $(R,D_{0})$.

\end{prop}
\begin{proof} We plug the singular expansion~(\ref{eq:singD}) of $D(x)$ into
Equation~(\ref{eq:Bexplicit}) and work out the undetermined
coefficients $B_i$. One can check that $B_{1}=2R^{2}
\Phi(R,D_{0})D_{1}$, which vanishes because $\Phi(x,D(x))=0$.

When computing $B_3$, it turns out that the values $D_2$ and $D_3$
are irrelevant because they appear in a term which contains a factor
$\Phi_z$, which by definition vanishes at  $(R,D_0)$. This
observation gives directly Equation~(\ref{eq:B3}). The fact $B_3\neq
0$ follows from applying the same argument as in the proof of
Proposition~\ref{proposition:D's}, that is, $4T_z - 2R T_{x z}<0$.
Then $B_3>0$ since it is the product of two negative numbers.
\end{proof}

\subsubsection{$\Phi$ is singular at $(R,
D_0)$}\label{sssec:Phi-singular}
In this case we assume that $T(x,z)$ is singular at $(R, D_0)$ and
that $\Phi_z(R, D_0)<0$. The situation where both $T(x,z)$ is
singular and $\Phi_z(R, D_0)=0$ is treated in the next subsection. We start with a technical lemma.

\begin{lemma}\label{lem:boundedT}
The function $T_{zz}$ is bounded at the singular point $(R, D_0)$.
\end{lemma}

\begin{proof} By differentiating Equation~(\ref{eq:phi}) with respect to $z$ we
obtain
$$
\Phi_z(x,z)= \frac{2}{x^{2}} T_{zz}(x,z)
-\frac{1}{1+z}-\frac{1}{(1+xz)^2}+1.
$$
Since $\Phi_z(R, D_0)<0$, we have
 $$
 {2\over R^2}T_{zz}(R,D_0) < {1\over 1+D_0}
+ {1\over (1+RD_0)^2}-1 < 1.
$$
Hence $T_{zz}(R,D_0) < R^2/2$.
\end{proof}

Let us consider now the singular expansions of $\Phi$ and $T$ in
terms of $Z=\sqrt{1-z/r(x)}$, where $r(x)$ is the dominant
singularity. Note that, by Equation~(\ref{eq:phi}), $\Phi$ and $T_z$
have the same  singular behaviour. By Lemma~\ref{lem:boundedT}, the
singular exponent $\alpha$ of the dominant singular term $Z^\alpha$
of $T_{zz}$ must be greater than $0$ and, consequently, the singular
exponent of $T_z$ and $\Phi$ is greater than $1$. As discussed
above, we only study the case where the singular exponent of
$T(x,z)$ is $5/2$ (equivalently, the singular exponent of
$\Phi(x,z)$ is $3/2$), which corresponds to several families of
three-connected graphs arising from maps. That is, we assume that
$T$ has a singular expansion of the
 form
$$
  T(x,z) = T_0(x) + T_2(x) Z^2 + T_4(x) Z^4 + T_{5}(x) Z^5 + O(Z^6),
$$
where $Z = \sqrt{1-z/r(x)}$, and the functions $r(x)$ and $T_i(x)$
are analytic in a neighborhood of $R$. Notice that $r(R)=D_0$. Since
we are assuming that the singular exponent is $5/2$, we have that
$T_5(R) \ne 0$.

We introduce now the Taylor expansion of the coefficients $T_i(x)$
at $R$. However, since we aim at computing the singular expansions
of $D(x)$ and $B(x)$ at $R$, we expand in even powers of
$X=\sqrt{1-x/R}$:
\begin{eqnarray}\label{eq-singT}
  T(x,z)   &= & T_{0,0}+ T_{0,2}X^2 + O(X^4) \\
               &&+\left(T_{2,0}+ T_{2,2}X^2 + O(X^4)\right)\cdot Z^2 \nonumber \\
               & &+\left(T_{4,0}+ T_{4,2}X^2 + O(X^4)\right)\cdot Z^4 \nonumber \\
               &&+\left(T_{5,0}+ T_{5,2}X^2 + O(X^4)\right)\cdot Z^5 +O(Z^6). \nonumber
\end{eqnarray}
Notice that $T_{5,0} = T_5(R) \ne 0$.
Similarly, we also consider the expansion of $\Phi$ given by
\begin{eqnarray}\label{eq-singPhi}
  \Phi(x,z)  &= & \Phi_{0,0}+ \Phi_{0,2}X^2 + O(X^4) \\
               & &+\left(\Phi_{2,0}+ \Phi_{2,2}X^2 + O(X^4)\right)\cdot Z^2 \nonumber \\
               & &+\left(\Phi_{3,0}+ \Phi_{3,2}X^2 + O(X^4)\right)\cdot Z^3 +O(Z^4),\nonumber
\end{eqnarray}
where $\Phi_{2,0}\neq 0$ because $\Phi_z(R, D_0)<0$.

The next result shows that $D_1=0$ and $D_3>0$. This was proved in \cite{bender} for the class of planar graphs, but there was no obvious reason explaining this fact. Now we see it follows directly from our general assumptions on $T(x,z)$, which are satisfied when $T(x,z)$ is the GF of 3-connected planar graphs.

\begin{prop}\label{pro:Dsing}
The function $D(x)$ admits the following singular expansion
$$
  D(x) =  D_0 + D_2 X^2 + D_3 X^3 + O(X^4),
$$
where $X = \sqrt{1-x/R}$. Moreover,
$$
  D_2 = D_0 \, \frac{P}{Q}-Rr', \qquad
  D_3 = -\frac{ 5 T_{5,0} (-P)^{3/2}}{R^2 \, Q^{5/2}} > 0,
$$
where $r'$ is the evaluation of the derivative $r'(x)$ at $x=R$, and
$P<0$ and $Q>0$ are given by
\begin{align*}
  P =\, \Phi_{0,2} = & -\frac{4T_{2,0}+ 2T_{2,2}}{R^2D_0}
  - \frac{2T_{2,0}r'}{RD_0^2}
  + \frac{Rr'}{1+D_0}
  - \frac{RD_0(D_0+(2+RD_0)Rr')}{(1+RD_0)^2}, \\
  Q =\, \Phi_{2,0} =& -\frac{4T_{4,0}}{R^2D_0}+\frac{D_0}{1+D_0}-\frac{2RD_0^2}{1+RD_0}
  + \frac{R^2D_0^3}{(1+RD_0)^2}.
\end{align*}
\end{prop}

\begin{proof} We consider Equation~(\ref{eq-singPhi}) as a power series
$\Phi(X,Z)$, where  $X = \sqrt{1-x/R}$ and $Z = \sqrt{1-z/D_0}$. We
look for a solution $Z(X)$ such that $\Phi(X, Z(X)) = 0$; we also
impose $Z(0)=0$, since $\Phi_{0,0}=\Phi(R,D_0)=0$. Define  $D(x)$ as
$$
  D(x) = r(x) (1-Z(X)^2),
$$
which satisfies  $\Phi(x, D(x))=0$ 
and $D(R)  = D_0$. By indeterminate coefficients we obtain
$$
  Z(X) = \pm\, \sqrt{\frac{-\Phi_{0,2}}{\Phi_{2,0}}} X +
         \frac{\Phi_{3,0}\,\Phi_{0,2}}{2\,{\Phi_{2,0}}^2} X^2 +
         O(X^3),
$$
where the sign of the coefficient in $X$ is determined later. Now we
use this expression and the Taylor series of the analytic function
$r(x)$ at $x=R$ to obtain the following singular expansion for
$D(x)$:
$$
  D(x) = D_0 + \left(D_0\frac{\Phi_{0,2}}{\Phi_{2,0}}-R r'\right)X^2
  \pm
  D_0\frac{(-\Phi_{0,2})^{3/2}\,\Phi_{3,0}}{{\Phi_{2,0}}^{5/2}}X^3+O(X^4).
$$
Observe in particular that the coefficient of $X$ vanishes. We
define $P=\Phi_{0,2}$ and $Q=\Phi_{2,0}$. The fact that $P<0$ and
$Q>0$ follows from the relations
$$
  \Phi_z = \frac{-1}{D_0} \Phi_{2,0},  \qquad
  \Phi_x = \frac{-1}{R} \Phi_{0,2} +
  \frac{r'}{D_0^2}\Phi_{2,0},
  $$
that are obtained by differentiating Equation~(\ref{eq-singPhi}). We
have  $\Phi_z<0$ by assumption, and  $\Phi_x>0$ following the proof
of Proposition~\ref{proposition:D's}.

The coefficient $D_3$  must have positive sign,  since $D''(x)$ is a
positive function and its singular expansion is $D_{xx}(x) =
3D_3(4R^2)^{-1} X^{-1} +O(1)$. The coefficients $\Phi_{i,j}$ in
Equation~(\ref{eq-singPhi}) are easily expressed in term of the
$T_{i,j}$, and a simple computation gives the result as claimed.
\end{proof}

\begin{prop}\label{prop:singB}
The function $B(x)$ admits the following singular expansion
$$
  B(x) = B_0 + B_2 X^2 + B_4 X^4 + B_5 X^5 + O(X^6),
$$
where $X = \sqrt{1-x/R}$. Moreover,
\begin{eqnarray*}
  B_0 &=& \frac{R^2}{2}\left(D_0+\frac{1}{2}D_0^2 \right) -\frac{1}{2}RD_0+\frac{1}{2}\log\left(1+RD_0\right)
   -\frac{1}{2}(1+D_0)\frac{R^3D_{0}^2}{1+RD_{0}}\\
  &&+T_{0,0}+\frac{1+D_{0}}{D_{0}}T_{2,0}, \\
  B_2 &=& \frac{R^2 D_0(D_0^2 R-2)}{2(1+R D_0)}+T_{0,2}-\left(2\frac{1+D_{0}}{D_0}+\frac{R r'}{D_0}\right)T_{2,0}, \\
  B_4 &=& \left(T_{0,4}+\frac{2R^3D_0^2-R^4D_0^4+2R^2D_0}{4(1+RD_0)^2} \right)+
  \left(\frac{1+D_0+r''}{D_0}\right)T_{2,0}+\frac{P^2}{Q}\frac{R^2 D_0}{4}\\
  && +\left(\frac{2R}{D_0}T_{2,0}+\frac{R^4D_0^2}{2(1+RD_0)^2}\right)r'
  +\frac{R^4}{4}\left(\frac{D_0}{1+D_0} -\frac{1}{(1+RD_0)^2}\right)(r')^2, \\
  B_5 &=& T_{5,0}\left(-\frac{P}{Q}\right)^{5/2} < 0 ,
\end{eqnarray*}
where $P$ and $Q$ are as in  Proposition~\ref{pro:Dsing}, and $r'$
and $r''$ are the derivatives of $r(x)$ evaluated at $x=R$.
\end{prop}

\begin{proof} Our starting point is Equation~(\ref{eq:Bexplicit}) relating
functions $D$, $B$ and $T$. We replace $T$ by the singular expansion
in Equation~(\ref{eq-singT}), $D$ by the singular expansion given in
Proposition~\ref{pro:Dsing},  and we set $x=X^2(1-R)$. The
expressions for $B_i$ follow by indeterminate coefficients.

When performing these computations we observe that the coefficients
$B_1$ and $B_3$ vanish identically, and that several simplifications
occur in the remaining expressions.
\end{proof}

\subsubsection{$\Phi$ has a branch-point and $T(x,z)$ is singular at
$(R, D_0)$}\label{se:criticalB}
This is the first critical situation, and corresponds to case (3.1)
in Theorem~\ref{th:mainresult}. To study this case we proceed
exactly as in the case where $\Phi$ is singular at $(R, D_0)$
(Section~\ref{sssec:Phi-singular}), except that now $\Phi_z(R,
D_0)=0$. It is easy to check that Lemma~\ref{lem:boundedT} still
applies (with the bound $T_{zz}(R, D_0) \leq R^2/2$). As done in the
previous section, we only take into consideration families of graphs
where the singular exponent of $T(x,z)$ is $5/2$ (equivalently, the
singular exponent of $\Phi(x,z)$ is $3/2$).
Equations~(\ref{eq-singT}) and~(\ref{eq-singPhi}) still hold, except
that now $\Phi_{2,0}=0$ because of the branch point at $(R, D_0)$.
This missing term is crucial, as we make clear in the following
analogous of Proposition~\ref{pro:Dsing}.
Notice that $\Phi_{3,0} \ne 0$ because of our assumptions on $T(x,z)$.

\begin{prop}\label{pro:Dsing-crit}
The function $D(x)$ admits the following singular expansion
$$
  D(x) =  D_0 + D_{4/3} X^{4/3} + O(X^{2}),
$$
where $X = \sqrt{1-x/R}$ and
$$
  D_{4/3} = -D_0\left(\frac{-\Phi_{0,2}}{\Phi_{3,0}}\right)^{2/3}.
$$
\end{prop}

\begin{proof} As in the proof of Proposition~\ref{pro:Dsing}, we consider a
solution $Z(X)$ of the functional equation $\Phi(X, Z(X))=0$, and
define $D(x)$ as $r(x)(1-Z(X)^2)$. However, the singular development
of $\Phi(x,z)$ is now
\begin{eqnarray*}
  \Phi(x,z)  &= & \Phi_{0,2}X^2 + O(X^4) \\
                &&+\left(\Phi_{2,2}X^2 + O(X^4)\right)\cdot Z^2 \nonumber \\
                &&+\left(\Phi_{3,0}+ \Phi_{3,2}X^2 + O(X^4)\right)\cdot Z^3 +
                 O(Z^4).\nonumber
\end{eqnarray*}
since $\Phi_{2,0} \ne 0$, the only way to get the necessary cancelations in $\Phi(X,Z(X))=$, is that the expansion of $Z(X)$ starts with a $Z^{2/3}$ term.
By indeterminate coefficients we get
$$
Z(X) =
\left(\frac{-\Phi_{0,2}}{\Phi_{3,0}}\right)^{2/3}X^{2/3}+O(X^{4/3}).
$$
To obtain the actual development of $D(x)$ we use the
equalities $D(x)=r(x)(1-Z(X)^2)$ and $r(R)=D_0$.
\end{proof}

Note that $X=\sqrt{1-x/R}$. Consequently the previous result implies
that the singular exponent of $D(x)$ is $2/3$. By using the explicit
integration of $B_y(x,y)$ of Equation~(\ref{eq:Bexplicit}), one can
check that the singular exponent of $B(x)$ is $5/3$ (the first
non-analytic term of $B(x)$ that does not vanish is $X^{10/3}$).
This implies that the subexponential term in the asymptotic of $b_n$
is $n^{-8/3}$, as claimed.

\subsection{Singularity analysis of $C(x,y)$ and $G(x,y)$}\label{subsec:singC}

The results in this section follow the same lines as those in the previous  section. They are technically simpler, since the analysis applies to functions of one variable, whereas the second variable $y$ behaves only as a parameter.
It generalizes the analysis in
Section~$4$ of \cite{gn} and Section~$3$ of \cite{SP}.

Let
$F(x)=xC'(x)$, which is the GF of rooted connected graphs. We know
that $F(x)=x\exp(B'(F(x))$. Then $\psi(u)=u\exp(-B'(u))$ is the
functional inverse of $F(x)$. Denote by $\rho$ the dominant
singularity of $F$. As for 2-connected graphs, there are two
possible sources for the singularity:

\begin{itemize}
  \item[(1)] There exists $\tau \in (0,R)$ (necessarily unique)
  such that $\psi'(\tau)=0$. We have a branch point and by the inverse
  function theorem $\psi$ ceases to be invertible at $\tau$. We have   $\rho = \psi(\tau)$.
  \item[(2)] We have $\psi'(u) \ne 0$ for all $u \in (0,R)$, and there
  is no obstacle to the analyticity of the inverse function.
  Then $\rho = \psi(R)$.
\end{itemize}

The critical case  where both sources for singularity coincide  is discussed at the end of this subsection. Notice that this happens precisely when $\psi'(R) = 0$.

Condition $\psi'(\tau)=0$ is equivalent to $B''(\tau) = 1/\tau$.
Since $B''(u)$ is increasing (the series $B(u)$ has positive
coefficients) and $1/u$ is decreasing, we are in case (1) if $B''(R)
> 1/R$, and in case (2) if $B''(R) < 1/R$. As we have already
discussed, series-parallel graphs correspond to case (1) and planar
graphs to case (2). In particular, if $B$ has singular exponent
$3/2$, like for series-parallel graphs, the function $B''(u)$ goes
to infinity when $u$ tends to $R$, so there is always a solution
$\tau < R$ satisfying $B''(\tau) = 1/\tau$. This explains why in
Theorem~\ref{th:mainresult} there is no case where $b_n$ has
sub-exponential growth $n^{-5/2}$ and $c_n$ has $n^{-7/2}$.

\begin{prop}\label{proposition:C's}
The value $S=RB''(R)$ determines the singular exponent  of $C(x)$
and~$G(x)$ as follows:
\begin{enumerate}
\item[(1)]  If $S>1$, then $C(x)$ and $G(x)$ admit the singular
expansions
\begin{eqnarray*}
  C(x) &=& C_0+C_2 X^2 + C_3 X^3 + O(X^4), \\
  G(x) &=& G_0+G_2 X^2 + G_3 X^3 + O(X^4),
\end{eqnarray*}
where $X=\sqrt{1-x/\rho}$, $\rho=\psi(\tau)$, and $\tau$ is the
unique solution to $\tau B''(\tau)=1$. We have
\begin{align*}
  C_0 &= \tau(1+\log \rho-\log \tau)+B(\tau), & C_2 &= -\tau, \\
  C_3 &= \frac{3}{2}\sqrt{\frac{2\rho \exp{\left(B'(\rho)\right)}}{\tau B'''(\tau)-\tau B''(\tau)^2+2B''(\tau)}}, \\
  G_0 &= e^{C_0},   \qquad G_2 = C_2e^{C_0},  & G_3 &=C_3e^{C_0}.
\end{align*}
\item[(2)] If $S<1$, then $C(x)$ and $G(x)$ admit the singular
expansions
\begin{align*}
  C(x) &= C_0+C_2 X^2 + C_4 X^4 + C_5 X^5 + O(X^6), \\
  G(x) &= G_0+G_2 X^2 + G_4 X^4 + G_5 X^5 + O(X^6),
\end{align*}
where $X=\sqrt{1-x/\rho}$, $\rho=\psi(R)$. We have
\begin{align*}
  C_0 &= \tau(1+\log \rho-\log R)+B_0,     & C_2 &= -R, \\
  C_4 &= -\frac{RB_4}{2B_4-R},             & C_5 &= B_5 \left(1-\frac{2B_4}{R}\right)^{-5/2}, \\
  G_0 &= e^{C_0},                          & G_2 &= C_2 e^{C_0}, \\
  G_4 &= \left(C_4+\frac{1}{2}{C_2}^2\right)e^{C_0},  & G_5 &= C_5 e^{C_0},
\end{align*}
where $B_0$, $B_4$ and $B_5$ are as in Proposition~\ref{prop:singB}.
\end{enumerate}
\end{prop}

\begin{proof} The two cases $S>1$ and $S<1$ arise from the previous discussion. In
case (1) we follow the proof  of Theorem~3.6 from~\cite{SP}, and in
case (2) the proof of  Theorem~1 from~\cite{gn}.

First, we obtain the singular expansion of $F(x)=xC'(x)$ near
$x=\rho$. This can be done by indeterminate coefficients in the
equality $ \psi( F(x) )= x = \rho(1-X^2)$, with $X=\sqrt{1-x/\rho}$.
The  expansion of $\psi$ can be either at $\tau=F(\rho)$ where it is
analytic, or at $R=F(\rho)$ where it is singular.

From the singular expansion of $F(x)$ we obtain $C_2$ and $C_3$ in
case (1), and $C_2$, $C_4$ and $C_5$ in case (2)  by direct
computation. To obtain $C_0$, however, it is necessary to compute
$$
  C(x) = \int_0^x \frac{F(t)}{t}\, dt,
$$
and this is done using the integration techniques developed
in~\cite{SP} and~\cite{gn}.
Finally, the coefficients for $G(x)$ are obtained directly from the
general relation $G(x) = \exp(C(x))$.
\end{proof}

To conclude this section we consider the critical case where both
sources of the dominant singularity $\rho$ coincide, that is, when
$\psi'(R)=0$. In this case $\psi$ is singular at $R$ because $R$ is
the singularity of $B(x)$, and at the same time the inverse $F(x)$
is singular at $\rho = \psi(R)$ because of the inverse function
theorem.
As we have shown before, this can only happen if $B(x)$ has
singular exponent $5/2$.

The argument is now as in the proof of Proposition \ref{pro:Dsing-crit}.
The singular development of $\psi(z)$
in terms of $Z=\sqrt{1-z/R}$ must be of the form
$$\psi(z) = \psi_{0} + \psi_2 Z^2 + \psi_3 Z^3 + O(Z^4), $$
where in addition $\psi_2$ vanishes due to $\psi'(R)=0$. A
similar analysis as that in Proposition \ref{pro:Dsing-crit} shows that
the singular exponent of $C(x)$ is $5/3$. Indeed, since
$\psi(F(x))=x=\rho(1-X^2)$, we deduce that the development of $F(x)$
in terms of $X=\sqrt{1-x/\rho}$ is
$$ F(x) = \rho + \left(\frac{-\rho^{5/3}}{\psi_3^{2/3}}\right)X^{4/3}
+ O(X^2).$$
Thus we obtain, by integration of $F(x)=xC'(x)$, that the singular
exponent of $C(x)$ is $5/3$, so that the subexponential term in the
asymptotic of $c_n$ is $n^{-8/3}$.  Since $G(x) = \exp (C(x))$, the
same exponents hold for $G(x)$ and $g_n$.

\section{Limit laws}\label{se:laws}

In this section we discuss parameters of random graphs from a closed
family whose limit laws  do not depend on the singular behaviour of
the GFs involved. As we are going to see, only the  constants
associated to the first two moments depend on the singular
exponents.

The parameters we consider are asymptotically either normal or
Poisson distributed. The number of edges, number of blocks, number
of cut vertices, number of copies of a fixed block, and number of
special copies of a fixed subgraph are all normal. On the other
hand, the number of connected components is Poisson. The size of the
largest connected component (rather, the number of vertices not in
the largest component) also follows a discrete limit law. A
fundamental extremal parameter, the size of the largest block, is
treated in the next section, where it is shown that the asymptotic
limit law depends very strongly on the family under consideration.

As in the previous section, let $\mathcal{G}$ be a closed family of
graphs. For a fixed value of $y$, let $\rho(y)$ be the dominant
singularity of $C(x,y)$, and let $R(y)$ be that of $B(x,y)$. We
write $\rho = \rho(1)$ and $R=R(1)$.  Recall that $B'(x,y)$ denotes
the derivative with respect to $x$.

When we speak of cases (1) and (2), we refer to the statement of
Proposition~\ref{proposition:C's}, which are exemplified,
respectively, by series-parallel and planar graphs. That is, in case
(1) the singular dominant term in $C(x)$ and $G(x)$ is
$(1-x/\rho)^{3/2}$, whereas in case (2) it is $(1-x/\rho)^{5/2}$.
Recall from the previous section that in case (1) we have $\rho(y) =
\tau(y)\exp{\left(-B'(\tau(y),y)\right)}$, where $\tau(y)
B''(\tau(y)) = 1$. In case (2) we have $\rho(y) =
R(y)\exp{\left(-B'(R(y),y)\right)}$.

\subsection{Number of edges}

The number of edges obeys a limit normal law, and the asymptotic
expression for the first two moments is always given in terms of the
function $\rho(y)$ for connected graphs, and in terms of $R(y)$ for
2-connected graphs.

\begin{theorem}\label{th:edges}
The number of edges in a random graph from $\mathcal{G}$ with $n$
vertices  is asymptotically normal, and the mean $\mu_n$ and
variance $\sigma_n^2$ satisfy
\begin{equation*}
\mu_n \sim \kappa n, \qquad \sigma_n^2 \sim \lambda n,
\end{equation*}
where
$$
  \kappa = -{\rho'(1) \over \rho(1)}, \qquad
  \lambda = -{\rho''(1) \over \rho(1)}  -{\rho'(1) \over \rho(1)}
    + \left( {\rho'(1) \over \rho(1)}\right)^2.
   $$
The same is true, with the same constants, for \emph{connected}
random graphs.

The number of edges in a random 2-connected graph from $\mathcal{G}$
with $n$ vertices  is asymptotically normal, and the mean $\mu_n$
and variance $\sigma_n^2$ satisfy
\begin{equation*}
\mu_n \sim \kappa_2 n, \qquad \sigma_n^2 \sim \lambda_2 n,
\end{equation*}
where
$$
  \kappa_2 = -{R'(1) \over R(1)}, \qquad
  \lambda_2 = -{R''(1) \over R(1)}  -{R'(1) \over R(1)}
    + \left( {R'(1) \over R(1)}\right)^2.
   $$
\end{theorem}

\begin{proof} The proof is as in \cite{gn} and \cite{SP}. In all cases the derivatives of
$\rho(y)$ and $R(y)$ are readily computed, and for a given family of
graphs we can compute the constants exactly.
\end{proof}
\subsection{Number of blocks and cut vertices}

Again we have normal limit laws but the asymptotic for the first two
moments depends  on which case we are. In the next statements we set
$\tau = \tau(1)$.

\begin{theorem}\label{th:blocs}
The number of blocks  in a random connected graph from $\mathcal{G}$
with $n$ vertices is asymptotically normal, and the mean $\mu_n$ and
variance $\sigma_n^2$ are linear in $n$. In case (1) we have
\begin{equation*}
\mu_n \sim \log(\tau/ \rho)\, n, \qquad
 \sigma_n^2 \sim \left(\log(\tau/ \rho) - {1 \over 1+\tau^2 B'''(\tau)}\right)\,
 n.
\end{equation*}
In case (2) we have
\begin{equation*}
\mu_n \sim \log(R/\rho) \,n, \qquad \sigma_n^2 \sim \log(R/\rho) \,
n.
\end{equation*}
The same is true, with the same constants, for arbitrary random
graphs.
\end{theorem}

\begin{proof} The proof for case (2) is as in \cite{gn}, and is based in an application of the Quasi-Powers Theorem. If $C(x,u)$ is the
generating function of connected graphs where now $u$ marks blocks,
then we have
\begin{equation}\label{eq:blo}
xC'(x,u) = x \exp\left(u B'(xC'(x,u))   \right),
\end{equation}
where derivatives are as usual with respect to $x$. For fixed $u$,
$\psi(t) = t\exp(-uB'(t))$ is the functional inverse of $xC'(x,u)$.
We know that for $u=1$, $\psi'(t)$ does not vanish, and the same is
true for $u$ close to 1 by continuity. The dominant singularity of
$C(x,u)$ is at $\sigma(u) = \psi(R) = R \exp(-uB'(R))$, and it is
easy to compute the derivatives $\sigma'(1)$ and $\sigma''(1)$ (see
\cite{gn} for details).

In case (1), Equation (\ref{eq:blo}) holds as well, but now the
dominant singularity is at $\psi(\tau)$. A routine (but longer)
computation gives the constants as claimed.
\end{proof}

\begin{theorem}\label{th:cut}
The number of cut vertices  in a random connected graph from
$\mathcal{G}$ with $n$ vertices is asymptotically normal, and the
mean $\mu_n$ and variance $\sigma_n^2$ are linear in $n$. In case
(1) we have
\begin{equation*}
\mu_n \sim \left(1- {\rho \over \tau}\right) n, \qquad \sigma_n^2
\sim
\left(\frac{(\tau-\rho)(\tau-2\tau\rho^2-\tau\rho-\rho+2\rho^3)}{\tau^2
\rho^2(1+\tau^2B'''(\tau))}-\left(\frac{\rho}{\tau}\right)^2
\right)n.
\end{equation*}
In case (2) we have
\begin{equation*}
\mu_n \sim \left(1 - {\rho \over R} \right) n, \qquad \sigma_n^2
\sim {\rho \over R} \left(1 - {\rho \over R} \right) n.
\end{equation*}
The same is true, with the same constants, for arbitrary random
graphs.
\end{theorem}

\begin{proof} If $u$ marks cut vertices in $C(x,u)$, then we have
$$xC'(x,u)=xu(\exp\left(B'(xC'(x,u))\right)-1)+x.$$
It follows that, for given $u$,
$$\psi(t) = \frac{t}{u(\exp{(B'(t))}-1)+1}$$
 is the inverse function of $xC'(x,u)$. In case (2) the dominant singularity $\sigma(u)$ is  at $\psi(R)$.
Taking into account that $\rho = R\exp(B'(R))$, the derivatives of
$\sigma$ are easily computed. In case (1) the singularity is at
$\psi(\tau(u))$, where $\tau(u)$ is given by  $\psi'(\tau(u))=0$. In
order to compute derivatives,
 we differentiate $\psi(\tau(u))=0$ with respect to $u$ and solve for
 $\tau'(u)$, and once more in order to get $\tau''(u)$. After several computations and simplifications using \texttt{Maple}, we get the values as claimed.
\end{proof}

\subsection{Number of copies of a subgraph}
Let $H$ be a fixed rooted graph from the class $\mathcal{G}$, with
vertex set $\{1,\ldots,h\}$ and root $r$. Following \cite{MSW}, we
say that $H$ \emph{appears}  in $G$ at $W \subset V(G)$ if (a) there
is an increasing bijection from $\{1,\ldots,h\}$ to $W$ giving an
isomorphism between $H$ and the induced subgraph $G[W]$ of $G$; and
(b) there is exactly one edge in $G$ between $W$ and the rest of
$G$, and this edge is incident with the root $r$.

Thus an appearance of $H$ gives a copy of $H$ in $G$ of a very
particular type, since the copy is joined to the rest of the graph
through a unique pendant edge. We do not know how to count the
number of subgraphs isomorphic to $H$ in a random graph, but we can
count very precisely the number of appearances.

\begin{theorem}\label{th:appear}
Let $H$ be a fixed rooted connected graph in $\mathcal{G}$ with $h$
vertices. Let $\nX_n$ denote the number of appearances of $H$ in a
random rooted connected graph from $\mathcal{G}$ with $n$ vertices.
Then $\nX_n$ is asymptotically normal and the mean $\mu_n$ and
variance
    $\sigma_n^2$ satisfy
\begin{equation*}
\mu_n \sim {\rho^h \over h!}\,n, \qquad \sigma_n^2 \sim \rho\, n,
\end{equation*}
\end{theorem}

\begin{proof} The proof  is as in \cite{gn}, and is based on the Quasi-Powers Theorem. If $f(x,u)$ is the generating
function of rooted connected graphs and $u$ counts appearances of
$H$ then, up to a simple term that does not affect the asymptotic
estimates, we have
\begin{equation*}
 f(x,u) = x \exp\left(B'(f(x,u)) + (u-1) {x^h\over h!}    \right).
\end{equation*}
The dominant singularity is computed through a change of variable,
and the rest of the computation is standard; see the proof of
Theorem 5 in \cite{gn} for details. For this parameter is no
difference between cases (1) and (2).
\end{proof}

Now we study appearances of a fixed 2-connected subgraph $L$ from
$\mathcal{G}$ in rooted connected graphs. An appearance of $L$ in
this case corresponds to a block with a labelling  order isomorphic
to $L$. Notice in this case  that an appearance can be anywhere in
the tree of blocks, not only as a terminal block.

\begin{theorem}\label{th:appearBlock}
Let $L$ be a fixed rooted 2-connected graph in $\mathcal{G}$ with
$\ell+1$ vertices. Let $\nX_n$ denote the number of appearances of
$L$ in a random connected graph from $\mathcal{G}$ with $n$
vertices. Then $\nX_n$ is asymptotically normal and the mean $\mu_n$
and variance
    $\sigma_n^2$ satisfy
\begin{equation*}
\mu_n \sim {R^\ell \over \ell!}\,n, \qquad \sigma_n^2 \sim {R^\ell
\over \ell!} \, n,
\end{equation*}
\end{theorem}

\begin{proof} If $f(x,u)$ is the generating function of rooted connected graphs
and $u$ counts appearances of $L$, then we have
\begin{equation*}
 f(x,u) = x \exp\left(B'(f(x,u)) + (u-1) {f(x,u)^\ell\over \ell!}    \right).
\end{equation*}
The reason as that each occurrence of $L$ is single out by
multiplying by $u$. Notice that $L$ has $\ell+1$ vertices by the
root bears no label. It follows that the inverse of $f(x,u)$ is
given by (for a given value of $u$) is
$$
   \phi(t)  = t \exp\left( -B'(t) - (u-1)t^\ell/\ell!\right).
$$
The singularity of $\phi(t)$ is equal to $R$, independently of $t$.
Since for $u=1$ we know that $\phi'(t)$ does not vanish, the same is
true for $u$ close to~1. Then the dominant singularity of $f(x,u)$
is given by
$$
\sigma(u) = \phi(R)= \rho\cdot \exp(-(u-1)R^\ell/\ell!),
$$
since $\rho = R \exp(-B'(R))$. A simple calculation gives
$$\sigma'(1) = -\rho {R^\ell \over \ell!}, \qquad \sigma''(1) = \rho
{R^{2\ell}\over \ell!^2}
$$
and the results follows easily as in the proof of
Theorem~\ref{th:edges}. Again, for this parameter  is no difference
between cases (1) and (2).
\end{proof}

\subsection{Number of connected components}
Our next parameter, as opposed to the previous one,  follows a
discrete limit law.
\begin{theorem}\label{th:components}
Let $\nX_n$ denote the number of connected components in a random
 graph $\mathcal{G}$ with $n$ vertices. Then  $\nX_n-1$ is distributed
asymptotically as a Poisson law of parameter $\nu$, where $\nu =
C(\rho)$.

As a consequence,  the probability that a random graph $\mathcal{G}$
is connected is asymptotically equal to $e^{-\nu}$.
\end{theorem}

\begin{proof} The proof is as in \cite{gn}. The generating function of graphs
with exactly $k$ connected components is $C(x)^k/k!$. Taking the
$k$-th power of the singular expansion of $C(x)$, we have $ [x^n]
C(x)^k  \sim k \nu^{k-1} [x^n] C(x)$.  Hence the probability that a
random  graphs has exactly $k$ components is asymptotically
$$
{ [x^n] C(x)^k / k! \over [x^n]G(x)} \sim {k \nu^{k-1} \over k!} \,
e^{-\nu} = {\nu^{k-1} \over (k-1)!} \, e^{-\nu}
$$
as was to be proved.
\end{proof}
\subsection{Size of the largest connected component}
Extremal parameters are treated in the next two sections. However,
the size of the largest component is easy to analyze and we include
it here. The notation ${M}_n$ in the next statement,
suggesting vertices \emph{missed} by the largest component,  is
borrowed from \cite{surfaces}. Recall that $g_n,c_n$ are the numbers
of graphs and connected graphs, respectively,  $R$ is the radius of
convergence of $B(x)$, and $C_i$ are the singular coefficients of
$C(x)$.

\begin{theorem}\label{th:largest-component}
Let ${L}_n$ denote the size of the largest connected
component in a random  graph $\mathcal{G}$ with $n$ vertices, and
let ${M}_n = {L}_n-n$. Then
$$
  \PP\left({M}_n = k\right) \sim  p_k = p\cdot  g_k {\rho^k \over k!} ,
$$
where $p$ is the probability of a random graph being connected.
Asymptotically, either $p_k \sim c\, k^{-5/2}$ or $p_k \sim c\,
k^{-7/2}$  as $k \to \infty$, depending on the subexponential term
in the estimate of $g_k$.

In addition, we have $\sum p_k = 1$ and
$\EE\left[{M}_n\right] \sim \tau$ in case $(1)$ and
$\EE\left[{M}_n\right] \sim R$  in case $(2)$.  In case (1)
the variance $\sigma^2({M}_n)$ does not exist and in case (2)
we have $\sigma^2({M}_n) \sim R + 2C_4 $.
\end{theorem}

\begin{proof} The proof is essentially the same as in \cite{surfaces}. For fixed
$k$, the probability that ${M}_n = k$ is equal to
$$
\binom{n}{k}\frac{c_{n-k} g_k}{g_n},
$$
since there are ${n \choose k}$ ways of choosing the labels of the
vertices not in the largest component, $c_{n-k}$ ways of choosing
the largest component, and $g_k$ ways of choosing the complement. In
case $(1)$, given the estimates
$$
    g_n \sim g \cdot n^{-5/2} \rho^{-n} n!, \qquad
    c_n \sim c \cdot n^{-5/2} \rho^{-n} n!,
$$
the estimate for $p_k$ follows at once (we argue similarly in each
subcase of $(2)$). Observe that $p = \lim c_n/g_n = c/g$.

For the second part of the statement notice that,
$$
\sum p_k  = p \sum g_k {\rho^k \over k!} = p\, G(\rho) = 1,
$$
since from Theorem~\ref{th:components} it follows that $p =
e^{-C(\rho)} = 1/G(\rho)$. To compute the moments notice that the
probability GF is  $f(u) = \sum p_k u^k = p G(\rho u)$. Then the
expectation is estimated as
$$
f'(1) = p\, \rho G'(\rho) = p\,  G(\rho) \rho C'(\rho) ,
$$
which correspond with $\tau$ in case $(1)$ and $R$ in case $(2)$,
since $G(x) = \exp C(x)$. For the variance we compute
$$
f''(1) + f'(1) - f'(1)^2 = \rho C'(\rho) + \rho^2 C''(\rho).
$$
In case (1) $\lim _{x\to \rho} C''(x) = \infty$, so that the
variance does not exist. In case (2) we have $\rho C'(\rho) = R$ and
$\rho^2 C''(\rho) = 2C_4$.
\end{proof}

\section{Largest block and 2-connected core}\label{se:bloc}

The problem of estimating the largest block in random maps has been
well studied. We recall that a map is a connected planar graph
together with a specific embedding in the plane. Moreover, an edge
has been oriented and marked as the root edge. Gao and
Wormald~\cite{GW99} proved that the largest block in a random map
with $n$ edges has almost surely $n/3$ edges, with deviations of
order~$n^{2/3}$. More precisely, if $\nX_n$ is the size of the
largest block, then
 $$
    \PP\left(|\nX_n - n/3| < \lambda(n) n^{2/3}\right) \to 1, \qquad \hbox{as $n
    \to \infty$},
 $$
where $\lambda(n)$ is any function going to infinity with $n$. The
picture was further clarified by Banderier et al.~\cite{airy}. They
found  that the largest block in random maps obeys a continuous
limit law, which is called by the authors the `Airy distribution of
the map type', and is closely related to a stable law of index
$3/2$. As we will see shortly, the Airy distribution also appears in
random planar graphs.

A useful technical device is to work with the 2-connected core,
which in the case of maps is the (unique) block containing the root
edge. For graphs it is a bit more delicate. Consider a connected
graph $R$ rooted at a vertex $v$. We would like to say that the core
of $R$ is the block containing the root, but if $v$ is a cut vertex
then there are several blocks containing $v$ and there is no clear
way to single out one of them. Another possibility is to say that
the $2$-connected core is the union of $2$-connected components the
blocks containing the root, but then the core is not in general a
$2$-connected graph.

The definition we adopt is the following. If the root is not a cut
vertex, then the \emph{core} is the unique block containing the
root. Otherwise, we say that the rooted graph is \emph{coreless}.
Let $C^{\bullet}(x,u)$ be the generating function of rooted
connected graphs, where the root bears no label, and $u$ marks the
size of the $2$-connected core. Then we have

\begin{equation*}
C^{\bullet}(x,u)= B'(uxC'(x))+ \exp(B'(xC'(x))-B'(xC'(x)),
\end{equation*}
where $C(x)$ and $B(x)$ are the GFs for connected and 2-connected
graphs, respectively. The first summand corresponds to graphs which
have a core, whose size is recorded through  variable $u$, and the
second one to coreless graphs. We rewrite the former equation as
$$
C^{\bullet}(x,u)= Q(uH(x)) + Q_L(x),
$$
where
 $$
 H(x)=xC'(x), \quad Q(x)=B'(x), \quad Q_L(x) =
 \exp(B'(xC'(x))-B'(xC'(x)).
 $$
With this notation, $Q_L(x)$ enumerates coreless graphs, and
$Q(uH(x))$ enumerates graphs with core. The asymptotic probability
that a graph is coreless is
$$
p_L= \lim_{n\rightarrow \infty} \frac{[x^{n}]Q_{L}(x)}
{[x^{n}]C'(x)} = 1-\lim_{n\rightarrow \infty} \frac{[x^{n}]Q(H(x))}
{[x^{n}]C'(x)}.
$$
The key point is that graphs with core fit into a composition scheme
$$
Q(uH(x)).
$$
This has to be understood as follows. A rooted connected graph whose
root is not a cut vertex is obtained from a 2-connected graph (the
core), replacing each vertex of the core by a rooted connected
graph. It is shown in \cite{airy} that such a composition scheme
leads either to a discrete law or to a continuous law, depending on
the nature of the singularities of $Q(x)$ and $H(x)$.

Our analysis for a closed class $\mathcal{G}$ is divided into two
cases. If we are in case (1) of Proposition~\ref{proposition:C's},
we say that the class $\mathcal{G}$ is \emph{series-parallel-like};
in this situation the size of the core follows invariably a discrete
law which can be determined precisely in terms of $Q(x)$ and $H(x)$.
If we are in case (2) we say that the class $\mathcal{G}$ is
\emph{planar-like}.  In this situation the size of the core has two
modes, a discrete law when the core is small, and a continuous Airy
distribution when the core has linear size. Moreover, for
planar-like classes, the size of the largest block follows the same
Airy distribution and  is concentrated around $\alpha n$ for a
computable constant~$\alpha$. The critical case, discussed at the
end of Section~\ref{se:asympt}, is not treated here.

\subsection{Core of series-parallel-like classes}
Recall that in case (1) of Proposition~\ref{proposition:C's} we have
$H(\rho)=\rho C'(\rho)=\tau $, where $\tau$ is the solution to the
equation $\tau B''(\tau)=1$. Since $RB''(R)>1$ and $uB''(u)$ is an
increasing function, we conclude that $H(\rho)<R$. This gives rise
to the so called \emph{subcritical} composition scheme. We refer to
the exposition in section IX.3 of
\cite{FlajoletSedgewig:analytic-combinatorics}.
The main result we use is Proposition IX.1 from
\cite{FlajoletSedgewig:analytic-combinatorics}, which is the
following.

\begin{prop}\label{prop:subcritical-case}
Consider the composition scheme $Q(uH(x))$. Let $R,\rho$ be the
radius of convergence of $Q$ and $H$, respectively. Assume that $Q$
and $H$ satisfy the subcritical condition $\tau = H(\rho)\leq R$,
and that $H(x)$ has a unique singularity at $\rho$ on its disk of
convergence with a singular expansion
$$H(x)=\tau -c_{\lambda}(1-z/\rho)^{\lambda}+o((1-z/\rho)^{\lambda}),$$
where $\tau, c_{\lambda}>0$ and $0<\lambda<1$.
Then the size of the $Q$-core follows a discrete limit law,
$$\lim_{n\rightarrow \infty}\frac{[x^nu^k]Q(uH(x))}{[x^n]Q(H(x))}=q_{k}.$$
The probability generating function $q(u) = \sum q_ku^k$ of the
limit distribution is
$$
q(u)={uQ'(\tau u) \over Q'(\tau)}.
$$
\end{prop}
The previous result applies to our composition scheme $Q(uH(x))$,
that is, to the family of rooted connected graphs that  have core.

\begin{theorem}\label{th:coreSP}
Let $\mathcal{G}$ be a series-parallel-like class, and let $\nY_n$
be the size of the 2-connected core in a random rooted connected
graph $\mathcal{G}$ with core and $n$ vertices.
 Then $\PP\left(\nY_n= k\right)$ tends to a limit $q_k$ as $n$ goes to infinity. The
probability generating function $q(u) = \sum q_k u^k$ is given by
$$q(u)=\tau u B''(u\tau).$$
The estimates of $q_k$ for large $k$ depend on the singular
behaviour of $B(x)$ near $R$ as follows, where $X=\sqrt{1-x/R}$:
\begin{itemize}
\item[(a)]If $B(x)= B_{0}+B_2 X^2+B_{3}X^3+O(X^4)$, then
$q_{k}\sim
\displaystyle\frac{3B_{3}}{4R\sqrt\pi}k^{-1/2}\left(\frac{\tau}{R}\right)^{k}$.
\\
\item [(b)] If $B(x)= B_{0}+B_2 X^2+B_{4}X^4 +
B_5X^5+O(X^6)$, then $q_{k}\sim
-\displaystyle\frac{5B_{5}}{2R\sqrt\pi}k^{-3/2}\left(\frac{\tau}{R}\right)^{k}$.
\end{itemize}

Finally, the probability of a graph being coreless is asymptotically
equal to $1- \rho/\tau$.
\end{theorem}

\begin{proof} We apply Proposition \ref{prop:subcritical-case} with $Q(x)=B'(x)$
and $H(x)=xC'(x)$. Since $\tau B''(\tau)=1$,  we have
$$q(u)=\frac{uB''(u\tau)}{B''(\tau)}=\tau u B''(u\tau),
$$
as claimed. The dominant singularity of $q(u)$ is at $u=R/\tau$. The
asymptotic for the tail of the distribution follow by the
corresponding singular expansions. In case (a) we have
$$B''(X)= {3B_3 \over 4R^2} X^{-1} + O(1).$$
In case (b) we have
$$B''(X)= {2B_4 \over R^2} + {5B_5 \over R^2} X + O(X^2).$$
By applying singularity analysis to $q(u)$, the result follows. We
remark that $B_3>0$ and $B_5 <0$, so that the multiplicative
constants are in each case positive.
\end{proof}

It is shown in \cite{kostas} that the largest block in series-parallel classes
is of order $O(\log n)$. This is to be expected given the exponential tails of the distributions in the previous theorem.

\subsection{Largest block of planar-like
classes}\label{se:largestbloc-planar}

In order to state our main result, we need to introduce the Airy
distribution. Its density is given by
\begin{equation}\label{eq:airy}
  g(x) = 2 e^{-2x^3/3} (x {\rm Ai}(x^2) - {\rm Ai}'(x^2)),
\end{equation}
where ${\rm Ai}(x)$ is the Airy function, a particular solution of
the differential equation $y''-x y = 0$. An explicit series
expansion is (see equation (2) in \cite{airy})
 $$
 g(x) = {1 \over \pi x} \sum_{n\ge1} (-3^{2/3} x)^n {\Gamma(1+2n/3)
 \over n!} \, \sin(-2n\pi/3).
 $$
A plot of $g(x)$ is shown in Figure~\ref{fig:airy}. We remark that
the left tail (as $x \to -\infty$) decays polynomially while the
right tail (as $x \to +\infty$) decays exponentially.
\begin{figure}[htb]
\centerline{\includegraphics[width=0.89\textwidth]{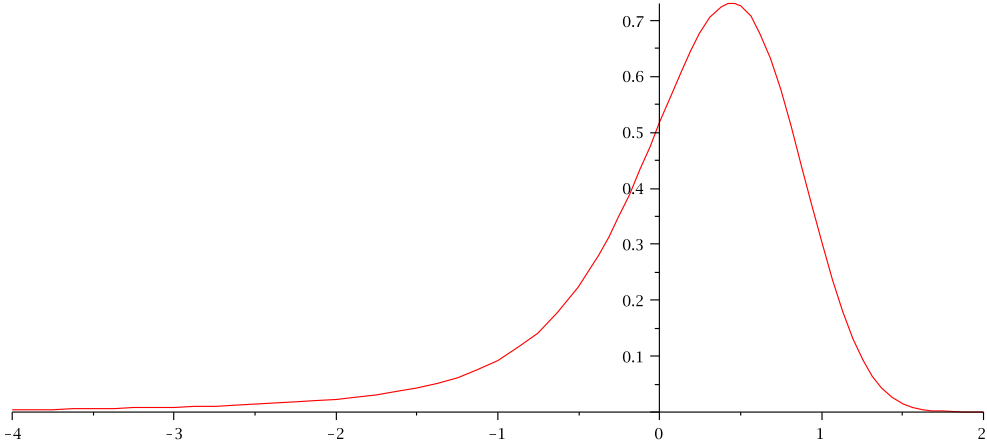}}
\bigskip \caption{The Airy distribution.} \label{fig:airy}
\end{figure}
We are in case (2) of Proposition~\ref{proposition:C's}. In this
situation we have $\rho = \psi(R)$ and $H(\rho) = R$, which is a
\emph{critical} composition scheme. We need  Theorem $5$ of
\cite{airy} and the discussion preceding it, which we rephrase in
the following proposition.

\begin{prop}\label{prop:critical-case}
Consider the composition scheme $Q(uH(x))$. Let $R,\rho$ be the
radius of convergence of $Q$ and $H$ respectively. Assume that $Q$
and $H$ satisfy the critical condition $H(\rho)=R$, and that $H(x)$
and $Q(z)$ have a unique singularity at $\rho$ and $R$ in their
respective discs of convergence. Moreover, the singularities of
$H(x)$ and $Q(z)$ are of type $3/2$, that is,
\begin{eqnarray*}
H(x)&=& H_0 + H_2 X^2 + H_3 X^3 + O(X^4),\\
Q(z)&=&Q_0+Q_{2}Z^2+Q_3Z^3 + O(Z^4),
\end{eqnarray*}
where $X=\sqrt{1-x/\rho}$, $Z=\sqrt{1-z/R}$. Let $\alpha_0$ and
$M_3$ be
$$
\alpha_0 = -\frac{H_0}{H_2}, \qquad M_3 = -\frac{Q_2
H_3}{R}+Q_3\alpha_0^{-3/2}.
$$
Then the asymptotic distribution of the size of the $Q$-core in
$Q(uH(z))$ has two different modes. With probability $p_s = -Q_2
H_3/(R M_3)$ the core has size $O(1)$, and with probability $1-p_s$
the core follows a continuous limit Airy distribution
concentrated at $\alpha_0 n$. More precisely, let $\nY_n$ be the
size of the $Q$-core of a random element of size $n$ of $Q(uH(z))$.
\begin{enumerate}
\item[(a)] For fixed $k$,
$$ \PP\left(\nY_n = k\right) \sim \frac{H_3}{M_3} k R^{k-1} [z^k]Q(z).$$
\item[(b)] For $k = \alpha_0 n + x n^{2/3}$ with $x = O(1)$,
$$ n^{2/3} \PP\left(\nY_n = k\right) \sim \frac{ Q_3 \alpha_0^{-3/2}}{M_3} c g(cx),
\quad c=\frac{1}{\alpha_0}\left( \frac{-H_2}{3 H_3}\right)^{2/3},$$
where $c g(cx)$ is the Airy distribution of parameter $c$.
\end{enumerate}
\end{prop}
In particular, we have $\EE\left[\nX_n\right] \sim \alpha n$. The
parameter $c$ quantifies in some sense the dispersion of the
distribution (not the variance, since the second moment does not
exist). Note that the asymptotic probability that the core has size
$O(1)$ is
$$ p_s = \sum^{\infty}_{k=0} \PP\left(\nX_n = k\right) \sim
   \frac{H_3}{M_3} \sum^{\infty}_{k=0} k R^{k-1} [z^k]Q(z) =
   \frac{H_3}{M_3} Q'(R) =  \frac{H_3}{M_3} \left(\frac{-Q_2}{R}\right),
$$
and that the asymptotic probability that the core has size
$\Theta(n)$ is
$$ {Q_3 \alpha_0^{-3/2} \over M_3} = 1-p_s.$$
Now we state the main result in this section. Recall that for a
planar-like class of graphs we have
$$
B(X)= B_{0}+B_2 X^2+B_{4}X^4 + B_5X^5+O(X^6),
$$
where $R$ is the dominant singularity of $B(x)$ and
$X=\sqrt{1-x/R}$.

\begin{theorem}\label{th:largest-block}
Let $\mathcal{G}$ be a planar-like  class, and let $X_n$ be the size
of the largest block in a random connected graph $\mathcal{G}$ with
$n$ vertices. Then
$$
    \PP\left(\nX_n = \alpha  n + x n^{2/3}\right) \sim n^{-2/3} c g(c x),
$$
where $$ \alpha = {R-2B_4 \over R},
 \qquad   c = \left({-2R \over 15B_5}\right)^{2/3},$$ and $g(x)$ is as in~(\ref{eq:airy}).
Moreover, the size of the second largest block is  $O(n^{2/3})$. In
particular, for the class of planar graphs we have
 $\alpha \approx 0.95982$ and $c \approx
 128.35169$.
\end{theorem}

\begin{proof} The composition scheme in our case is $B'(uxC'(x))$. In the notation
of the previous proposition, we have $Q(x)=B'(x) $ and $H(x) =
xC(x)$.

The size of the core is obtained as a direct application of
Proposition \ref{prop:critical-case}. The exact values for planar
graphs have been computed using the known singular expansions for
$B(x)$ and $C(x)$ given in the appendix of \cite{gn}.

For the size of the largest block, one can adapt an argument from
\cite{airy}, implying that the probability that the core has linear
size while not being the largest block tends to 0 exponentially
fast. It follows that the distribution of the size of the largest
block is exactly the same as the distribution of the core in the
linear range.
\end{proof}

The main conclusion is that for planar-like classes of  graphs (and
in particular for planar graphs) there exists a unique largest block
of linear size, whose expected value is asymptotically $\alpha n$
for some computable constant $\alpha$. The remaining block are of
size $O(n^{2/3})$. This is in complete contrast with series-parallel
graphs, where we have seen that there are only blocks of sublinear
size.

\subsection*{Remark.}
An observation that we need later, is that if the largest block $L$
has $N$ vertices, the it is uniformly distributed among all the
2-connected graphs in the class. This is because the number of
graphs of given size whose largest block is $L$ depends only on the
number of vertices of $L$, and not on its isomorphism type.

\medskip
We can also analyze the size of the largest block for graphs with a
given edge density, or average degree.
We state a precise result for planar graphs, which is probably the most interesting one.

\begin{theorem}
For $\mu \in (1,3)$, the largest block in random planar graphs with $n$ vertices and $\lfloor \mu n\rfloor$ edges follows asymptotically an Airy law with computable parameters $\alpha(\mu)$ and $c(\mu)$.
\end{theorem}

\begin{proof}
As discussed in \cite{gn}, we
choose a value $y_0>0$ depending on $\mu$ such that, if we give
weight $y_0^k$ to a graph with $k$ edges, then only graphs with $n$
vertices and $\mu n$ edges have non negligible weight. If $\rho(y)$
is the radius of convergence of $C(x,y)$ as usual, the right choice
is the unique positive solution $y_0$ of
\begin{equation}\label{eq:mu-y}
    -y  \rho'(y)/\rho(y) = \mu,
\end{equation}
Then we work with the generating function $xC(x,y_0)$ instead of
$xC'(x)$. Again we have a critical composition scheme, and as in the proof of Theorem \ref{th:largest-block}, the size of the largest block follows asymptotically an Airy law.
\end{proof}

Figure \ref{fig:plot2conn} shows a plot of the main
parameter $\alpha(\mu)$ for planar graphs and $\mu \in (1,3)$. When $\mu \to 3^-$
we see that $\alpha(\mu)$ approaches 1; the explanation is that  a planar triangulation is 3-connected and hence has a unique block. When $\mu \to 1^+$, $\alpha(\mu)$ tends to 0, in this case because  the largest block in a tree is just an edge.

\begin{figure}[htb]
\centerline{\includegraphics[width=0.8\textwidth]{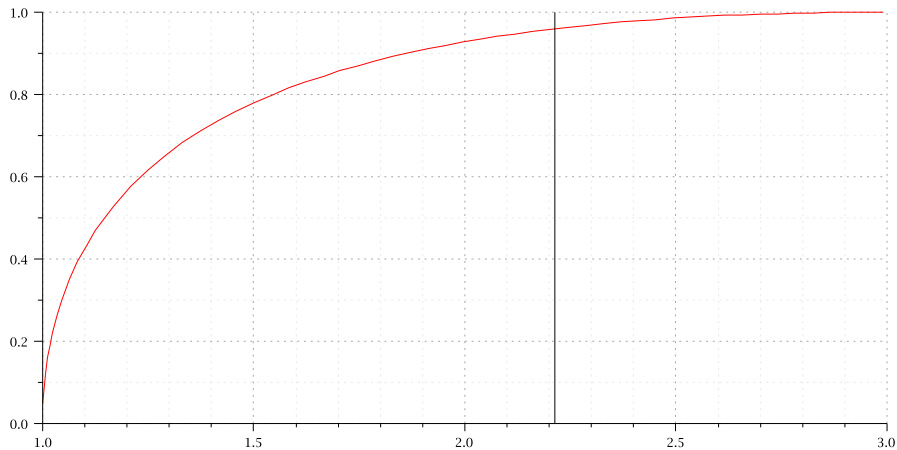}}\bigskip
 \caption{Size of largest block for planar graphs
 with $\mu n$ edges, $\mu \in (1,3)$. The ordinate gives the value
 $\alpha(\mu)$ such that the largest block has size $\sim
 \alpha(\mu) n$. The value at $\kappa$ is $0.9598$
  as in Theorem~\ref{th:largest-block}. } \label{fig:plot2conn}
\end{figure}

\section{Largest $3$-connected component}

Let us recall that, given a 2-connected graph $G$, the 3-connected
components of $G$ are those 3-connected graphs that are the support
of h-networks in the network decomposition of $G$.

We have seen in Theorem~\ref{th:largest-block} that the largest
block in a random graph from a planar-like class is almost surely of
linear size, and it is unique. In this section we prove a similar
result for  the largest 3-connected component in random connected
graphs with $n$ vertices. Again we obtain a limit Airy law, but the
proof is more involved. 
There are three main technical issues we
need to address:

\begin{enumerate}
\item We start with a connected graph $G$. We know from Theorem~\ref{th:largest-block}
that the largest block $L$ of $G$ is distributed according to an Airy law. We show that the largest
3-connected component $T$ of $L$ is again Airy distributed. Thus we
have to \emph{concatenate} two Airy laws, and we show that we obtain
another Airy law with computable parameters. Our proof is based on
the fact that the sum of two independent stable laws of the same
index $\alpha$ (recall that the Airy law corresponds to a particular
stable law of index $3/2$) is again an Airy law with computable
parameters.
In order to illustrate this step, we prove a result of independent
interest: given a random planar map with $m$ edges, the size of the
largest 3-component is Airy distributed with expected value $n/9$.

\item We also need to analyze the number of edges in the largest block $L$
of a connected graph. The number of vertices of $L$ is Airy
distributed with known parameters~\ref{th:largest-block}. On the
other hand, the number of edges in 2-connected graphs with $N$
vertices is asymptotically normally distributed with expected value
$\kappa_2 N$ (see Theorem~\ref{th:edges}). Thus we have to study a
parameter normally distributed (number of edges) within the largest
block, whose size (number of vertices) follows an Airy law. We show
that the composition of these two limit laws gives rise to an Airy
law for the number of edges in the largest block, again with
computable parameters.

\item The analysis of the largest block in random connected graphs is
in terms of the number of vertices, but the analysis for the largest
3-connected component of a 2-connected graph is necessarily in terms
of the number of edges. Thus we need a way to relate both models.
This is done through a technical lemma that shows that two
probability distributions on 2-connected graphs with $m$ edges are
asymptotically equivalent. This is the content of
Lemma~\ref{the:asympt-equal}.
\end{enumerate}

Our main result is the following. We state it for planar graphs,
since this is the most interesting case and we can give explicitly
the parameters, but it holds more generally for planar-like classes
of graphs.

\begin{theorem}\label{th:3conn-main}
Let $\nX_n$ be the number of vertices in the largest 3-connected
component of a random connected planar graph with $n$ vertices. Then
$$
    \PP\left(\nX_n = \alpha_2n + xn^{2/3}\right) \sim  n^{-2/3} c_2g(c_2x),
    $$
where $\alpha_2 \approx 0.7346$ and $c_2\approx 3.14596$ are
computable constants. Additionally, the number of edges in the
largest 3-connected component of a random connected planar graph
with $n$ vertices also follows asymptotically an Airy law with
parameters $\alpha_3 \approx 1.7921$ and $c_3 \approx 1.28956$.
\end{theorem}

The rest of the section is devoted to the proof of the theorem. The
next three subsections address the technical points discussed above.
We remark that for series-parallel-like classes there is no linear  3-connected component, just as for 2-connected components.

\subsection{Largest 3-connected component in random planar maps}

Recall that a planar map (we say just a map) is a connected planar
graph together with a specific embedding in the plane. The size of
largest $k$-components in several families of maps was thoroughly
studied in~\cite{airy}. Denote by $M(z)$, $B(z)$ and $C(z)$ the
ordinary GFs associated to maps, $2$-connected maps and
$3$-connected maps, respectively; in all cases, $z$ marks edges. Let
$\nnL_n$ be the random variable, defined over the set of maps with
$n$ edges, equal to the size of the largest 2-connected component.
Let $\nnT_m$ be the random variable, defined over the set of
2-connected maps with $n$ edges, equal to the size of the largest
3-connected component.

In \cite{airy} it is shown the following result:

\begin{theorem}\label{thm:maps-airy}
The distribution of both $\nnL_n$ and $\nnT_m$ follows
asymptotically an Airy law, namely
\begin{eqnarray}\label{eq:Airy-maps}
\PP\left(\nnL_n = a_1 n + xn^{2/3}\right)  &\sim& n^{-2/3}c_1g(c_1x), \\
\PP\left(\nnT_m = a_2 m + ym^{2/3}\right) &\sim&
m^{-2/3} c_2g(c_2 y),\nonumber
\end{eqnarray}
where $g(z)$ is the map Airy distribution, $a_1=1/3$,
$c_1=3/4^{2/3}$, $a_2=1/3$, and $c_2 = 3^{4/3}/4$.
\end{theorem}

\begin{proof}
Here is a sketch of the proof. In both cases, the distribution
arises from a critical composition scheme of the form $\frac{3}{2}
\,\circ \,\frac{3}{2}$. The distribution of $\nnL_n$ is given by the
scheme $B\left(z(1+M(z))^2\right)$, which reflects the fact that a
map is obtained by gluing a map at each corner of a $2$-connected
map. In the second case, the result is obtained from the composition
scheme $C\left(B(z)/z-2\right)$, which reflects the fact that a
2-connected map is obtained by replacing each edge of a 3-connected
map by a non-trivial 2-connected map (to complete the picture one
must take also into account series and parallel compositions, but
these play no role in the analysis of the largest 3-connected
component, see \cite{census2} for details).
\end{proof}

Let $\nX_n$ be the random variable equal to the size of the largest
$3$-connected component in maps of $n$ vertices. In order to get a a
limit law for $\nX_n$, we need a more detailed study of stable laws.
In particular, Airy laws are particular examples of stable laws of
index $3/2$. Our main reference is the forthcoming
book~\cite{Nolan}. The result we need is Proposition $1.17$, which
appears in \cite[Section 1.6]{Nolan}. We rephrase it here in a form
convenient for us.

\begin{prop}\label{prop:nolan}
Let $\nY_1$ and $\nY_2$ be independent Airy distributions, with
density probability functions $c_1g(c_1x)$ and $c_2g(c_2x)$. Then
 $\nY_1+\nY_2$ follows an Airy distribution with
density probability function $c g(cx)$, with
$c=\left(c_1^{-3/2}+c_2^{-3/2}\right)^{-2/3}$.
\end{prop}

\begin{proof} We use the notation as in~\cite{Nolan}. A stable
law is characterized by its \emph{stability factor} $\alpha\in
[0,2)$, its \emph{skewness} $\beta\in [-1,1]$, its \emph{factor
scale} $\gamma>0$, and its \emph{location parameter} $\delta\in
\mathbb{R}$. A stable random variable with this parameters is
written in the form $S(\alpha,\beta,\gamma, \delta;1)$ (the constant
$1$ refers to the type of the parametrization; we only deal with
this type). Proposition $1.17$ in~\cite{Nolan} states that if
$S_1=S(\alpha,\beta_1,\gamma_1, \delta_1;1)$ and
$S_2=S(\alpha,\beta_2,\gamma_2, \delta_2;1)$ are independent random
variables, then $S_1+S_2=S(\alpha,\beta,\gamma, \delta, 1)$, with

\begin{equation}\label{eq:stable-parameters}
\beta=\frac{\beta_1\gamma_1^{\alpha}+\beta_2\gamma_2^\alpha}
{\gamma_1^{\alpha}+\gamma_2^\alpha},\,\,\gamma^{\alpha}=\gamma_1^{\alpha}+\gamma_2^\alpha,\,\,
\delta=\delta_1+\delta_2.
\end{equation}

Let us identify the Airy distribution with density probability
function $c g(cx)$ within the family of  stable laws as defined. By
definition, the stability factor is equal to $3/2$. Additionally,
$\beta=-1$: this is the unique value that makes that a stable law
decreases exponentially fast (see Section $1.5$ of~\cite{Nolan}).
The value of the location parameter $\delta$ coincides with the
expectation of the random variable, hence $\delta=0$ (see
Proposition $1.13$). Finally, the factor scale can be written in the
form $\gamma_0/c$, for a suitable value of $\gamma_0$, the one which
corresponds with the normalized Airy distribution with density
$g(x)$. Since $\nY_1=S(3/2,-1,\gamma_0/c_1,0;1)$ and
$\nY_2=S(3/2,-1,\gamma_0/c_2,0;1)$, the result follows from
(\ref{eq:stable-parameters}).
\end{proof}

\begin{theorem}\label{th:maps}
The size $\nX_n$ of the largest $3$-connected component in a random
map with $n$ edges follows asymptotically an Airy law of the form
$$
\PP\left(\nX_n = a n + zn^{2/3}\right) \sim n^{-2/3}
c g(c z),
$$
where  $g(z)$ is the Airy distribution and
\begin{equation}
a=a_1 a_2=1/9, \qquad
c=\left(\left(\frac{c_1}{a_2}\right)^{-3/2}+c_2^{-3/2}a_1\right)^{-2/3}\approx
1.71707.
\end{equation}
\end{theorem}

\begin{proof}
Let us estimate  $n^{2/3}\PP\left(\nX_n = a n +
zn^{2/3}\right)$ for large $n$. Considering the possible
values size of the largest $2$-connected component, we obtain
\begin{equation*}
n^{2/3}\PP\left(\nX_n = a n +
zn^{2/3}\right)=n^{2/3}\sum_{m=1}^{\infty}
\PP\left(\nnL_n=m\right)\PP\left(\nnT_m=a
n+zn^{2/3}\right).
\end{equation*}
In the previous equation we have used the fact that the largest
$2$-connected component is distributed uniformly among all
$2$-connected maps with the same number of edges; this is because
the number of ways a 2-connected map $M$ can be completed to a map
of given size depends only on the size of $M$.

Notice that $\nX_n$ and $\nnT_m$ are \emph{integer} random
variables, hence the previous equation should be written in fact as
\begin{equation*}
n^{2/3}\PP\left(\nX_n = \lfloor a n + z
n^{2/3}\rfloor \right)=n^{2/3}\sum_{m=1}^{\infty}
\PP\left(\nnL_n=m\right)\PP\left(\nnT_m=\lfloor
a n+zn^{2/3}\rfloor\right).
\end{equation*}
Let us write $m=a_1 n+x n^{2/3}$. Then $an+zn^{2/3}=a_2 m+y
m^{2/3}+o\left(m^{2/3}\right)$, where $y=a_1^{-2/3}(z-a_2 x)$.
Observe that when we vary $m$ in one unit, we vary $x$ in $n^{-2/3}$
units. Let $x_0 = (1-a_1 n)n^{-2/3}$, so that $a_1 n + x_0 n^{2/3} =
1$ is the initial term in the sum. The previous sum can be written
in the form
\begin{equation*}
n^{2/3}\sum_{x=x_0+\ell n^{-2/3}}
 \PP\left(\nnL_n=a_1
n+xn^{2/3}\right)\PP\left(\nnT_m=a_2
m+\alpha_1^{-2/3}(z-a_2 x)m^{2/3}\right).
\end{equation*}
where the sum is for all values  $\ell \ge 0$. From
Theorem~\ref{thm:maps-airy} it follows that
\begin{eqnarray*}
&&n^{2/3}\sum_{x=x_0 + \ell n^{-2/3}}
 \PP\left(\nnL_n=a_1
n+xn^{2/3}\right)\PP\left(\nnT_m=a_2 m+\alpha_1^{-2/3}(z-a_2 x)m^{2/3}\right)\\
&\sim& n^{2/3}\sum_{x=x_0 + \ell n^{-2/3}}
 n^{-2/3}c_1g(c_1 x)\,\,m^{-2/3} c_2 g\left(c_2 a_1^{-2/3}(z-a_2
 x)\right)\\
&\sim& \frac{1}{n^{2/3}}\sum_{x=x_0 + \ell n^{-2/3}}
 c_1g(c_1 x)  \,\,c_2 a_1^{-2/3} g\left(c_2 a_1^{-2/3}(z-a_2
 x)\right).
\end{eqnarray*}
In the last equality we have used that $m^{-2/3}=(a_1
n)^{-2/3}(1+o(1)).$ Now we approximate by an integral:
\begin{eqnarray*}
&&n^{-2/3}\sum_{x=x_0 + \ell n^{-2/3}}
 c_1g(c_1 x) \,\,  c_2 a_1^{-2/3} g\left(c_2 a_1^{-2/3}(z-a_2
 x)\right)\\
&\sim& \int_{-\infty}^\infty c_1 g(c_1 x) \,\,c_2a_1^{-2/3}
g\left(c_2a_1^{-2/3}(z-a_2 x)\right) dx
\end{eqnarray*}
The previous estimate holds uniformly for $x$ in a bounded interval.
Now  we set $a_2 x=u$, and with this change of variables we get
$$\int_{-\infty}^\infty \frac{c_1}{a_2} g\left(\frac{c_1}{a_2} u\right)\,\, c_2a_1^{-2/3}
g\left(c_2a_1^{-2/3}(z-u)\right) du.$$
This convolution can be interpreted as a sum of stable laws with
parameter $3/2$ in the following way. Let $\nY_1$ and $\nY_2$ be
independent random variables with densities $\frac{c_1}{a_2}
g\left(\frac{c_1}{a_2} u\right)$ and $c_2a_1^{-2/3}
g\left(c_2a_1^{-2/3}(z-u)\right)$, respectively. Then, the previous
integral is precisely
$\PP\left(\nY_1+\nY_2=z\right)$, and the result
follows from Proposition~\ref{prop:nolan}.
\end{proof}

\subsection*{Remark.}
The previous theorem can be obtained, alternatively, using the
machinery developed in~\cite{airy}. The two composition schemes
$B(zM(z)^2)$ and $C(B(z)/z-2)$ can be composed algebraically into a
single composition scheme $C(B(zM(z)^2)/z-2$. This is again a
critical scheme with exponents $3/2$ and an Airy law follows from
the general scheme in~\cite{airy}. The parameters can be computed
using the singular expansions of $M(z)$, $B(z)$, $C(z)$ at their
dominant singularities which are, respectively, equal to $1/12$,
$4/27$ and $1/4$. We have performed the corresponding computations
in complete agreement with values obtained in Theorem~\ref{th:maps}.
We have chosen the present proof since the same ideas are used later
in the case of graphs, where no algebraic composition seems
available.

\subsection{Number of edges in the largest block of a connected graph}

As discussed above, we have a limit Airy law $\nX_n$ for the number
of vertices in the largest block $L$ in a random connected planar
graph. In order to analyze the largest 3-connected component of $L$,
we need to express $\nX_n$ in terms of the number of edges. This
amounts to combine the limit Airy law with a normal limit law,
leading to slightly modified Airy law. The precise result is the
following.
\begin{theorem}\label{th:edges-largestbloc}
Let $\nZ_n$ be the number of \emph{edges} in the largest block of a
random connected planar graph with $n$ vertices. Then
$$
    \PP\left(\nZ_n = \kappa_2\alpha n + zn^{2/3}\right) \sim n^{-2/3}{c\over\kappa_2} g\left({c\over\kappa_2}z\right),
    $$
here $\alpha$ and $c$ are as in Theorem~\ref{th:largest-block}, and
$\kappa_2\approx 2.26288$ is the constant for the expected number of
edges in random 2-connected planar graphs, as in
Theorem~\ref{th:edges}.
\end{theorem}

\begin{proof}
Let $\nX_n$ be, as in Theorem~\ref{th:largest-block}, the number of
vertices in the largest block. In addition, let $\nY_N$ be the
number of edges in a random 2-connected planar graph with $N$
vertices. Then
\begin{equation}\label{suma}
\PP\left(\nZ_n = \kappa_2\alpha n  + zn^{2/3}\right)
= \sum_{x=x_0+\ell n^{-2/3}} \PP\left(\nX_n = \alpha n +
xn^{2/3}\right) \PP\left(\nY_{\alpha n + xn^{2/3}} =
\kappa_2 \alpha n + zn^{2/3}\right),
\end{equation}
with the same convention for the index of summation as in the
previous section.

Since $\nY_N$ is asymptotically normal (Theorem~\ref{th:edges})
$$
\PP\left(\nY_N=\kappa_2 N + y N^{1/2}\right) \sim
N^{-1/2} h(y),
$$
here $h(y)$ is the density of normal law suitably scaled. If we take
$N = \alpha n + xn^{2/3}$, then $$\kappa_2 N = \kappa_2 \alpha N +
\kappa_2xn^{2/3}.$$As a consequence, the significant terms in the
sum in (\ref{suma}) are concentrated around $\alpha n +
(z/\kappa_2)n^{2/3}$, within a window of size $N^{1/2} =
\Theta(n^{1/2})$. Thus we can conclude that
$$
\PP\left(\nZ_n =\kappa_2\alpha n + zn^{2/3} \right)
\sim {1 \over \kappa_2} \PP\left(\nX_n = \lfloor \alpha n +
(z/\kappa_2)n^{2/3} \rfloor\right) \sim n^{-2/3}
{c\over\kappa_2} g\left({c\over\kappa_2}z\right),
$$
where $c$ is the constant in Theorem~\ref{th:largest-block}. The
factor $1 \over \kappa_2$ in the middle arises since in $\lfloor
\kappa_2\alpha n + zn^{2/3} \rfloor$ we have steps of length
$n^{-2/3}$, whereas in $\lfloor \alpha n + (z/\kappa_2) n^{2/3}
\rfloor$ they are of length $n^{-2/3}/\kappa_2$.
\end{proof}

\subsection{Probability distributions for 2-connected graphs}

In this section we study several probability distributions defined
on the set of 2-connected graphs of $m$ edges. The first
distribution $\nX^m_n$ (in fact, a family of probability
distributions, one for each $n$) models the appearance of largest
blocks with $m$ edges in random connected graphs with $n$ vertices.
The second one is a weighted distribution where each 2-connected
graph with $m$ edges receives a weight according to the number of
vertices, and for which it is easy to obtain an Airy law for the
size of the largest 3-connected component. We show that these two
distributions are asymptotically equivalent in a suitable range. In
particular, it follows that the Airy law of the latter distribution
also occurs in the former one. We start by defining precisely both
distributions. We use capital letters like $\nX^m_n$ and $\nY_m$ to
denote random variables whose output is a graph in the corresponding
universe of graphs, so that our distributions are associated to
these random variables. We find this convention more transparent
than defining the associated probability measures.

Let $n,m$ be fixed numbers, and let $\mathcal{C}^m_n$ denote the set
of connected graphs on $n$ vertices such that their largest block
$L$ has $m$ edges. The first probability distribution $\nX^m_n$ is
the distribution of $L$ in a graph of $\mathcal{C}^m_n$ chosen
uniformly at random. That is, if $B$ is a 2-connected graph with $m$
edges, and $\mathcal{C}^B_n\subseteq \mathcal{C}^m_n$ denotes the
set of connected graphs that have $L=B$ as the largest block, then
\[
\PP\left(\nX^m_n = B\right) =
\frac{|\mathcal{C}^B_n|}{|\mathcal{C}^m_n|}.
\]
Let $m$ be a fixed number. The second probability distribution
$\nY_m$ assigns to a 2-connected graph $B$ of $m$ edges and $k$
vertices the probability
\begin{equation}\label{eq:Ym}
\PP\left(\nY_m = B\right) = \frac{R^{k}}{k!}
\frac{1}{[y^m]B(R,y)},
\end{equation}
where $R$ is the radius of convergence of the exponential generating
function $B(x)=B(x,1)$ enumerating 2-connected graphs.
It is clear that $[y^m]B(R,y)$  
is the right normalization factor.

Now we state precisely what we mean when we say that these two
distributions are asymptotically equivalent in a suitable range. In
what follows, $\alpha$ and $\kappa_2$ are the multiplicative
constants of the expected size of the largest block and the expected
number of edges in a random connected graph (see Theorems
\ref{th:largest-block} and \ref{th:edges}). We denote by $V(G)$ and
$E(G)$ the set of vertices and edges of a graph $G$.

\begin{lemma}\label{the:asympt-equal}
Fix positive values $\bar{y},\bar{z}\in\mathbb{R}^+$. For fixed $m$,
let $I_n$ and $I_k$ denote the intervals

\begin{eqnarray*}
I_k &=& \left[\frac{1}{\kappa_2} m -\bar{z} m^{1/2},
\frac{1}{\kappa_2} m + \bar{z} m^{1/2}\right],\\
I_n &=& \left[\frac{1}{\alpha\kappa_2} m -\bar{y} m^{2/3},
\frac{1}{\alpha\kappa_2} m + \bar{y} m^{2/3}\right].
\end{eqnarray*}
Then the probability distributions $\nY_m$ and $\nX^m_n$, for $n\in
I_n$, are asymptotically equal on graphs with $k\in I_k$ vertices,
with uniform convergence for both $k$ and $n$. That is, there exists
a function $\epsilon(m)$ with $\lim_{m\to \infty} \epsilon(m)=0$
such that, for every 2-connected graph $B$ with $m$ edges and $k\in
I_k$ vertices, and for every $n\in I_n$, it holds that
\[
  \left|\frac{\PP\left(\nX^m_n = B\right)}{\PP\left(\nY_m = B\right)} - 1\right| < \epsilon(m).
\]
\end{lemma}

\begin{proof}
Fix $y\in[-\bar{y}, \bar{y}]$, and let $n = \lfloor
\frac{1}{\alpha\kappa_2} m + y m^{2/3}\rfloor \in I_n$. First, we
prove that $\nX^m_n$ and $\nY_m$ are concentrated on graphs with
$k=\frac{1}{\kappa_2} m+O(m^{1/2})$ vertices, that is,
\[\PP\left(\frac{1}{\kappa_2}m-z m^{1/2} \leq |V(\nX^m_n)|
            \leq \frac{1}{\kappa_2}m+z m^{1/2}\right) \]
goes to $1$ when $z, m\to \infty$, and the same is true for $\nY_m$.
Then, we show that $\nX^m_n$ and $\nY_m$ are asymptotically
\emph{proportional} for graphs on $k\in I_k$ vertices. A direct
consequence of both facts is that $\nX^m_n$ and $\nY_m$ are
asymptotically \emph{equal} in $I_k$, since the previous results are
valid for arbitrarily large $\bar{z}$.

We start by considering the probability distribution $\nY_m$.
If we add (\ref{eq:Ym}) over all the $b_{k,m}$ 2-connected graphs
with $k$ vertices and $m$ edges, we get
\[ \PP\left(|V(\nY_m)|=k\right) =  b_{k,m} \frac{R^{k}}{k!} \frac{1}{[y^m]B(R,y)}.\]
On the one hand, the value $[y^m]B(R,y)$ is a constant that does not
depend on $k$. On the other hand, since the numbers $b_{k,m}$
satisfy a local limit theorem (the proof is the same as
in~\cite{gn}), it follows that the numbers $b_{k,m} R^k/k!$ follow a
normal distribution concentrated at $k=1/\kappa_2$ on a scale
$m^{1/2}$, as desired.

We show the same result for $\nX^m_n$. Let $B$ be a 2-connected
graph with $k$ vertices and $m$ edges. We write the probability that
a graph drawn according to $\nX^m_n$ is $B$ as a conditional
probability on the largest block $\nL_n$ of a random connected graph
of $n$ vertices. In what follows, $v\left(\nL_n\right)$ and
$e\left(\nL_n\right)$ denote, respectively, the number of vertices
and edges of $\nL_n$.
\begin{eqnarray*}
\PP\left(\nX^m_n=B\right) &= &\PP\left(\nL_n=B
~\mid~ e(\nL_n)=m\right)\\&=&
\frac{\PP\left(\nL_n=B,\,\,
e(\nL_n)=m\right)}{\PP\left(e(\nL_n)=m\right)}=
\frac{\PP\left(\nL_n=B\right)}{\PP\left(e(\nL_n)=m\right)}.
\end{eqnarray*}
Note that in the last equality we drop the condition
$e\left(\nL_n\right)=m$ because it is subsumed by $\nL_n=B$.

The probability that the largest block $\nL_n$ is $B$ is the same
for all 2-connected graphs on $k$ vertices. Hence, if $b_{k}$
denotes the number of 2-connected graphs on $k$ vertices, we have
\[ \PP\left(\nX^m_n=B\right) = \frac{1}{b_k}\frac{\PP\left(v(\nL_n)=k\right)}{\PP\left(e(\nL_n)=m\right)},\]
If we sum over all the $b_{k,m}$ 2-connected graphs $B$ with $k$
vertices and $m$ edges, we finally get the probability of $\nX^m_n$
having $k$ vertices,
\[ \PP\left(|V(\nX^m_n)|=k\right) = \frac{b_{k,m}}{b_k}\frac{\PP\left(v(\nL_n)=k\right)}{\PP\left(e(\nL_n)=m\right)}.
\]
For fixed $n, m$, the numbers
$\PP\left(V(\nL_n)=k\right)$ follow an Airy
distribution of scale $n^{2/3}$ concentrated at $k_1=\alpha n$ (see
Section~\ref{se:largestbloc-planar}), and the numbers $b_{k,m}/b_k$
are normally distributed around $k_2=m/\kappa_2$ on a scale
$m^{1/2}$. The choice of $n$ makes $k_1$ and $k_2$ coincide but for
a lower-order term $O(m^{2/3})$; hence, it follows that
$\PP\left( |V(\nX^m_n)|=k\right)$ is concentrated at
$k_2=m/\kappa_2$ on a scale $m^{1/2}$, as desired.

Now that we have established concentration for both probability
distributions, we just need to show that they are asymptotically
proportional in the range $k=m/\kappa_2+O(m^{1/2})$. This is easy to
establish by considering asymptotic estimates. Indeed, we have
\[
\PP\left(\nX^m_n=B\right) =
\frac{1}{b_k}\frac{\PP\left(v(\nL_n)=k\right)}{\PP\left(e(\nL_n)=m\right)},
\]
and since $\PP\left(v(\nL_n)=k\right)$ is Airy
distributed in the range $k=m/\kappa_2+O(m^{2/3})$ and $b_k\sim
b\cdot k^{-7/2} R^{-k} k!$, it follows that
\[
\PP\left(\nX^m_n=B\right) \sim b  \cdot k^{7/2}
\frac{R^{k}}{k!}
\frac{g(x)}{\PP\left(e(\nL_n)=m\right)},
\]
where $x$ is defined as $(k-\alpha n) n^{-2/3}$ and $g(x)$ is the
Airy distribution of the appropriate scale factor. Let us compare it
with the exact expression for the probability distribution $\nY_m$,
that is,
\[
\PP\left(\nY_m=B\right) = \frac{R^{k}}{k!}
\frac{1}{[y^m]B(R,y)}.
\]
Clearly, both expressions coincide in the high order terms $R^k$ and
$1/k!$. The remaining terms are either constants like $b$,
$\PP\left(e(\nL_n)=m\right)$ and $[y^m]B(R,y)$, or
expressions that are asymptotically constant in the range of
interest. This is the case for  $k^{7/2}$, which is asymptotically
equal to $((1/\kappa_2)m)^{7/2}$. And also for $g(x)$, which is
asymptotically equal to $g(y(\alpha \kappa_2)^{2/3})$, since
$x=(k-\alpha n)n^{-2/3}$ and $n=\frac{1}{\alpha k_2}m + ym^{2/3}$
implies that
\begin{align*}
 x &= \left( \frac{1}{k_2}m + O(m^{1/2}) - \frac{\alpha}{\alpha k_2}m + ym^{2/3} \right)n^{-2/3} \\
  &=\left( ym^{2/3} \right)\left(\frac{m}{\alpha \kappa_2}\right)^{-2/3} + o(1) \\
  &= y(\alpha \kappa_2)^{2/3} + o(1)
\end{align*}
in the given range. Hence, both distributions are asymptotically
proportional in the given range.

Thus, we have shown the result when $n$ is linked to $m$ by
$n=\frac{1}{\alpha\kappa_2}m + ym^{2/3}$, for any $y$. Clearly,
uniformity holds when $y$ is restricted to a compact set of
$\mathbb{R}$, like $[-\bar{y}, \bar{y}]$.
\end{proof}

\subsection{Proof of the main result}
In order to prove Theorem~\ref{th:3conn-main}, we have to
concatenate two Airy laws. The first one is the number of edges in
the largest block, given by Theorem~\ref{th:edges-largestbloc}. The
second is the number of edges in the largest 3-connected component
of a random 2-connected planar graph with a given number of edges.
This is a gain an Airy law produced by the composition scheme
$T_z(x,D(x,y))$, which encodes the combinatorial operation of
substituting each edge of a 3-connected graph by a network (which is
essentially a 2-connected graph rooted at an edge). However, this
scheme is relative to the variable $y$ marking edges. In order to
have a legal composition scheme we need to take a fixed value of
$x$. The right value is $x=R$, as shown by
Lemma~\ref{the:asympt-equal} Indeed, taking $x=R$ amounts to weight
a 2-connected graph $G$ with $m$ edges with $R^k/k!$, where $k$ is
the number of vertices in $G$. Thus the relevant composition scheme
is precisely $T_z(R,uD(R,y))$, where $u$ marks the size of the
3-connected core. Formally, we can write it as the scheme
$$
    C(uH(y)),  \qquad H(y) = D(R,y), \qquad C(y) = T_z(R,y).
$$
The composition scheme $T_z(R,D(R,y))$ is critical with exponents
$3/2$, and an Airy law appears. In order to compute the parameters
we need the expansion of $D(R,y)$ at the dominant singularity $y=1$,
which is of the form
\begin{equation}\label{eq:Di1}
D(R,y) = \widetilde{D}_0 + \widetilde{D}_2 Y^2 + \widetilde{D}_3
Y^3+ O(Y^4),
\end{equation}
and $Y =\sqrt{1-y}$. The different $\widetilde{D}_i$ can be obtained
in the same way as in Proposition \ref{pro:Dsing}.

\begin{prop}\label{prop:edges-core-3con}
Let $\nW_m$ be the number of edges in the largest 3-connected
component of a 2-connected planar graph with $m$ edges, weighted
with $R^k/k!$, where $k$ is the number of vertices. Then
$$
    \PP\left(\nW_m = \beta n + zn^{2/3}\right) \sim n^{-2/3}c_2 g\left(c_2z\right),
    $$
where $\beta = -\widetilde{D}_0/\widetilde{D}_2 \approx 0.82513$ and
$c_2 = {-\widetilde{D}_2 / \widetilde{D}_0} \left({-\widetilde{D}_2
/ 3\widetilde{D}_3} \right)^{2/3} \approx 2.16648$, and the
$\widetilde{D}_i$ are as in Equation~(\ref{eq:Di1}).
\end{prop}

\begin{proof}The proof is a direct application of the methods in Theorem~\ref{th:largest-block}.
\end{proof}

\paragraph{Proof of Theorem \ref{th:3conn-main}.}

Recall that $\nX_n$ is the number of vertices in the largest
3-connected component of a random connected planar graph with $n$
vertices. This variables aries as the composition of two random
variables we have already studied. First we consider $\nZ_n$ as in
Theorem~\ref{th:edges-largestbloc}, which is the number of edges in
the largest block, and then $\nW_m$ as in
Proposition~\ref{prop:edges-core-3con}.

The main parameter turns out to be $\alpha_2 = \mu \beta (\kappa_2
\alpha)$, where
\begin{enumerate}
  \item $\alpha$ is for the expected number of vertices in the largest block;
  \item $\kappa_2$ is for the expected number of edges in 2-connected graphs;
  \item $\beta$ is for the expected number of edges in the largest 3-connected component;
  \item $\mu$ is for the expected number of vertices in 3-connected graphs weighted according to $R^k/k!$ if $k$ is the number of vertices.
\end{enumerate}
The constant in 1. and 3. correspond to Airy laws, and the constants
in 2. and 4. to normal laws.

Let $\nY_n$ be the number of edges in the largest 3-connected
component of a random connected planar graph with $n$ vertices
(observe that our main random variable $\nX_n$ is linked directly to
$\nY_n$ after extracting a parameter normally distributed like the
number of vertices). Then
$$
\PP\left(\nY_n = \beta \kappa_2 \alpha n + z
n^{2/3}\right)=\sum_{m=1}^{\infty}
\PP\left(\nZ_n=m\right)\PP\left(\nW_m= \beta
\kappa_2 \alpha n+zn^{2/3}\right).
$$
This convolution can be analyzed in exactly the same way as in the
proof of Theorem~\ref{th:maps}, giving rise to a limit Airy law with
the parameters as claimed.

Finally, in order to go from $\nY_n$ to $\nX_n$ we need only to
multiply the main parameter by $\mu$ and adjust the scale factor. To
compute $\mu$ we need the dominant singularity $\tau(x)$ of the
generating function $T(x,z)$ of 3-connected planar graphs, for a
given value of $x$ (see Section 6 in \cite{degrees}). Then
$$
\mu = -R \tau'(R) / \tau(R).
$$
Given that the inverse function $r(z)$ is explicit (see Equation
(25)~in~\cite{degrees}), the computation is straightforward.

\section{Minor-closed classes}\label{se:examples}

In this section we apply the machinery developed so far to analyse
families of graphs closed under minors. A class of graphs
$\mathcal{G}$  is minor-closed if whenever a graph is in
$\mathcal{G}$ all its minors are also in $\mathcal{G}$. Given a
minor-closed class $\mathcal{G}$, a graph $H$ is an excluded minor
for $\mathcal{G}$ if $H$ is not in $\mathcal{G}$ but every proper
minor is in~$\mathcal{G}$. It is an easy fact that a graph is in
$\mathcal{G}$ if and only if it does not contain as a minor any of
the excluded minors from $\mathcal{G}$. According to the fundamental
theorem of Robertson and Seymour, for every minor-closed class the
number of excluded minors is finite~\cite{RobSey}. We use the
notation $\mathcal{G} =\Ex(H_1, \cdots, H_k)$ if $H_1,\dots,H_k$ are
the excluded minors of $\mathcal{G}$. If all the $H_i$ are
3-connected, then $\Ex(H_1, \cdots, H_k)$ is a closed family. This
is because if none of the 3-connected components of a graph $G$
contains a forbidden, the same is true for $G$ itself.

In order to apply our results we must know which connected graphs
are in the set $\Ex(H_1, \cdots, H_k)$. There are several results in
the literature of this kind. The easiest one is $\Ex(K_4)$, which is
the class of series-parallel graphs. Since a graph in this class
always contains a vertex of degree at most two, there are no
3-connected graphs.  Table~\ref{tab:excludedminor} contains several
such results, due to Wagner, Halin and others (see Chapter X
in~\cite{diestel2}). The proofs make systematic use of Tutte's
wheels theorem (consult Chapter $3$ of~\cite{Diestel}, for
instance): a 3-connected graph can be reduced to a wheel by a
sequence of deletions and contractions of edges, while keeping it
3-connected.

{\renewcommand{\arraystretch}{1.1}
\begin{table}[htb]
\renewcommand{\arraystretch}{1.2}
\begin{center}
\begin{tabular}{|l|l|l|}
  \hline
  Ex. minors & $3$-connected graphs & Generating function $T(x,z)$ \\\hline
  $\displaystyle K_{4}$ & $\emptyset$ & $0$ \\
  $W_{4}$ & $K_{4}$ & $z^{6}x^4/4!$\\
  $K_{5}-e$ & $K_{3},K_{3,3},K_{3}\times K_{2},\{ W_{n}\}_{n\geq 3}$&
  $70z^9x^6/6!-\frac{1}{2}x(\log(1-z^2x)+2z^2x+z^4x^2)$ \\
   $K_5,K_{3,3}$ & Planar 3-connected & $T_p(x,z)$ \\
  $K_{3,3}$ & Planar 3-connected, $K_5$ & $T_p(x,z) + z^{10}x^5/5!$  \\
    $K_{3,3}^+$ & Planar 3-connected, $K_5, K_{3,3}$ & $T_p(x,z) + z^{10}x^5/5! + 10z^9x^6/6!$ \\
  \hline
\end{tabular}\bigskip
\caption{Classes of graphs defined from one excluded minor.
$T_p(x,z)$ is the GF of planar 3-connected graphs.}
\label{tab:excludedminor}
\end{center}
\end{table}
}
 The 3-connected graphs when excluding $W_5$ and the triangular
prism  take longer to describe. For $K_3 \times K_2$ they are: $K_5,
K_5^-,\{ W_{n}\}_{n\geq 3}$, and the family $G_\Delta$ of graphs
obtained from $K_{3,n}$ by adding any number of edges to the part of
the bipartition having 3 vertices. For $W_5$ they are: $K_4,K_5$,
the family $G_\Delta$, the graphs of the octahedron and the cube
$Q$, the graph obtained from $Q$ by contracting one edge, the graph
$L$ obtained from $K_{3,3}$ by adding two edge in one of the parts
of the bipartition, plus all the 3-connected subgraphs of the former
list. Care is needed here for checking that all 3-connected graphs
are included and for counting how many labellings each graph has.

Once we have the full collection of 3-connected graphs, we have
$T(x,z)$ at our disposal. For the family of wheels we have a
logarithmic term (see the previous table) and for the family
$G_\Delta$ it is a simple expression involving $\exp(z^3 x)$. We can
then apply the machinery developed in this paper and compute the
generating functions $B(x,y)$ and $C(x,y)$. For the last three
entries in Table~\ref{tab:excludedminor}, the main problem is
computing $B(x,y)$ and this was done in \cite{gn} and \cite{k33};
these correspond to the planar-like case. In the remaining cases
$T(x,z)$ is either analytic or has a simple singularity coming from
the term $\log(1-xz^2)$, and they correspond to the
series-parallel-like case.

In Table~\ref{tab:constants} we present the fundamental constants
for the classes under study. For a given class $\mathcal{G}$ they
are: the growth constants $\rho^{-1}$ of graphs in $\mathcal{G}$;
the growth constant $R^{-1}$  of 2-connected graphs in
$\mathcal{G}$; the asymptotic probability $p$ that a random graph in
$\mathcal{G}$ is connected; the constant $\kappa$ such that $\kappa
n$ is the asymptotic expected number of edges for graphs in
$\mathcal{G}$ with $n$ vertices; the analogous constant $\kappa_2$
for 2-connected graphs in $\mathcal{G}$; the constant $\beta$ such
that $\beta n$ is the asymptotic expected number of blocks for
graphs in $\mathcal{G}$ with $n$ vertices; and the constant $\delta$
such that $\delta n$ is the asymptotic expected number of cut
vertices for graphs in $\calG$ with $n$ vertices.
The values in Table~\ref{tab:constants} have been computed with
\texttt{Maple} using the results in sections \ref{se:asympt}
and~\ref{se:laws}.

{\renewcommand{\arraystretch}{1.2}
\begin{table}[htb]
\small
$$
\renewcommand{\arraystretch}{1.2}
\begin{array}{|l|c|c|c|c|c|c|c|}
\hline
\hbox{Class} & \rho^{-1} & R^{-1} & \kappa & \kappa_2 & \beta & \delta & p \\
\hline
\Ex(K_4) & 9.0733& 7.8123& 1.61673& 1.71891 &  0.149374& 0.138753& 0.88904\\
\Ex(W_4) & 11.5437& 10.3712& 1.76427& 1.85432 &  0.107065& 0.101533& 0.91305\\
\Ex(W_5) &  14.6667& 13.5508& 1.90239& 1.97981 &  0.0791307& 0.0760808& 0.93167\\
\Ex(K_5^-) & 15.6471& 14.5275& 1.88351& 1.95360 &  0.0742327& 0.0715444& 0.93597\\
\Ex(K_3 \times K_2) & 16.2404& 15.1284&1.92832& 1.9989 & 0.0709204 & 0.0684639 & 0.93832\\
\hbox{Planar} &  27.2269& 26.1841& 2.21327& 2.2629 &  0.0390518&
0.0382991& 0.96325 \\
\Ex(K_{3,3}) & 27.2293& 26.1866& 2.21338& 2.26299 &  0.0390483& 0.0382957& 0.963262\\
\Ex(K_{3,3}^+) &   27.2295& 26.1867& 2.21337& 2.26298 & 0.0390481&
0.0382956& 0.963263 \\
  \hline
\end{array}
$$\bigskip
\caption{Constants for a given class of graphs: $\rho$ and $R$ are
the radius of convergence of $C(x)$ and $G(x)$, respectively;
constants $\kappa,\kappa_2,\beta,\delta$ give, respectively, the
first moment of the number of: edges, edges in 2-connected graphs,
blocks and cut vertices; $p$ is the probability of connectedness. }
\label{tab:constants}
\end{table}
}

\section{Critical phenomena} \label{se:critical}

We have seen that the estimates for the number of planar graphs with
$\mu n$ edges have the same shape for all values $\mu \in (1,3)$.
This is also the case for series-parallel graphs, where $\mu \in
(1,2)$ since maximal graphs in this class have only $2n-3$ edges. It
is natural to ask if there are classes in which there is a
\textit{critical phenomenon}, that is, a different behaviour
depending on the
 edge density. We have not found such phenomenon for `natural'
classes of graphs, in particular those defined in terms of forbidden
minors. But we have been able to construct examples of critical
phenomena  by a suitable choice  of the family $\mathcal{T}$ of
3-connected graphs, as we now explain.

Let $\mathcal{T}$ a family of 3-connected graphs whose function
$T(x,z)$ has a singularity on $z$ of exponent $5/2$. We have seen that the
singular exponents of the associated functions $B(x)$, $C(x)$ and $G(x)$ depend
on the existence of branch-points before $T$ becomes singular. We
have obtained families of graphs for which the singular exponents of
$B(x)$, $C(x)$ and $G(x)$ depend on the particular value of $y_0$.

Now we have two sources for the main singularity of $B(x,y)$ for a
given value of $y$: either (a) it comes from the singularities of
$T(x,z)$; or (b) it comes from a branch point of the equation
defining $D(x,y)$. For planar graphs the singularity always comes
from case~(a), and for series-parallel graphs always from case~(b).
If there is a value $y_0$ for which the two sources coalesce, then
we get a different singular exponent depending on whether $y<y_0$ or
$y>y_0$. The most important consequence in this situation  is that
there is a critical edge density $\mu_0$, such that below $\mu_0$
the largest block has linear size, and above $\mu_0$ it has
sublinear size, or conversely.

Here are some examples.

\begin{itemize}
  \item If $\mathcal{T}$ is the family of \emph{3-connected cubic planar
graphs}, then $B(x,y)$ has singular exponent $5/2$ when
$y<y_0\approx 0.07422$, and  $3/2$ when $y>y_0$. The corresponding
critical value for the number of edges is $\mu_0 \approx 1.3172$.

  \item If $\mathcal{T}$ is the family of \emph{planar triangulations
(maximal planar graphs)}, then $B(x,y)$ exponent $3/2$ when
$y<y_0\approx 0.4468 $, and  $5/2$ when $y>y_0$. The corresponding
critical value for the number of edges is $\mu_0 \approx 1.8755$.

  \item This example shows that more than one critical value may
occur. This is done by adding a single dense graph to the family
$\mathcal{T}$ in the last example. Let $\mathcal{T}$ be the family
of \emph{triangulations plus the exceptional graph $K_6$}. Then
there are two critical values $y_0 \approx 0.4469$ and $y_1 \approx
108.88$, and the corresponding critical edge densities are $\mu_0
\approx 1.8756$ and $\mu_1 \approx 3.4921$. This last value is close
to $7/2$; this is the maximal edge density, which is approached by
taking many copies of $K_6$ glued along a common edge. It turns out
that $B(x,y)$ has exponent $3/2$ when $y<y_0$, $5/2$ when
$y_0<y<y_1$, and again $3/2$ for $y_1 < y$.
\end{itemize}

\bibliography{biblio-3-connex}
\bibliographystyle{abbrv}

\end{document}